\documentclass[reqno, 11pt]{amsart}
\usepackage{amsaddr}
\usepackage{comment}
\usepackage{amsmath} 
\usepackage{amssymb}
\usepackage{amsthm} 
\usepackage{caption}
\usepackage{cutwin}
\usepackage{graphicx}
\usepackage{subcaption}
\usepackage{enumerate}
\usepackage{float}
\usepackage{tabularx} 
\usepackage[dvipsnames, hyperref]{xcolor}
\usepackage{graphicx} 
\usepackage[margin=1in,letterpaper]{geometry} 
\usepackage[backend=biber,style=numeric,sorting=nyt,maxbibnames=99]{biblatex}
\usepackage[final]{hyperref} 
\hypersetup{
    colorlinks=true,
    linkcolor=BlueViolet,
    citecolor=JungleGreen,  
    urlcolor=cyan
    }
\usepackage{tikz}
\usetikzlibrary{arrows.meta,
                positioning,
                quotes
                }
\usepackage{wrapfig}
\usepackage{cleveref}

\newtheorem{theorem}{Theorem}[section]
\newtheorem{lemma}[theorem]{Lemma}

    \theoremstyle{definition}
\newtheorem{definition}[theorem]{Definition}
\newtheorem{assumption}{Assumption}
    \theoremstyle{remark}
\newtheorem{remark}[theorem]{Remark}

\def\R{\mathbb{R}}

\def\N{\mathbb{N}}

\def\x{\textbf{x}}
\def\ex{\mathbb{E}}
\def\ep{\epsilon}

\def\calL{\mathcal{L}}
\def\cL{\mathcal{L}}

\def\ra{\rightarrow}

\newcommand\numberthis{\addtocounter{equation}{1}\tag{\theequation}}

\def\sqra{\stackrel{\square}{\ra}}
\def\lra{\stackrel{L^2}{\ra}}

\keywords{Reaction--diffusion equations, graphons, stochastic processes, dynamical systems. }

\begin{document}

\title{Graphon Limits of Graph Reaction--Diffusion Equations}
\date{\today} 
\author{Edith J. Zhang}
\address{
    Department of Applied Physics and Applied Mathematics, Columbia University in the City of New York, NY, 10027, United States of \, America
    \\ Current Affiliation: 
  Department of Mathematics, University of California, Los Angeles, CA, 90095, United States of America
}
\email{{\tt edith@math.ucla.edu} (corresponding author)}

\author{James Scott}
\address{Department of Applied Physics and Applied Mathematics, Columbia University in the City of New York, NY, 10027, United States of  America \\
Current Affiliation: Department of Mathematics and Statistics, Auburn University, Auburn,  Alabama, 36849, United States of America}
\email{{\tt james.m.scott@auburn.edu}}

\author[Q. Du]{Qiang Du}
\address{Department of Applied Physics and Applied Mathematics, and Data Science Institute, Columbia University in the City of New York, NY, 10027, United States of  America}
\email{{\tt qd2125@columbia.edu}}

\begin{abstract}
    A graph reaction--diffusion (RD) equation is a system of differential equations that is defined on the nodes of a graph. 
    Consider a sequence of growing graphs that converges in cut norm to a limiting graphon. We show that the solutions of the sequence of graph RD equations converge in $L^p$ norm, for $p \in [1,\infty]$, to the solution of a limiting nonlocal RD equation, which we call a graphon RD equation. Furthermore, we show a large numbers result for a stochastic particle process that consists of a random walk and a birth-death process on graphs. For a sequence of graphs that converge in cut norm to a limiting graphon, the sequence of stochastic processes converges in probability to the solution of the graphon RD equation. 
\end{abstract}

\maketitle

\tableofcontents

%%%%%%%%%%%%%%%%%%%%%%%%%%%%%%

\section{Introduction} 

Classical reaction--diffusion (RD) equations originated to describe the evolution of a concentration of a substance in chemical reactions in which particles diffuse (i.e., travel) and react (i.e., appear and disappear). They arise in fields such as biology, ecology, chemistry, and physics~\cite{britton1986reaction, cantrell2004spatial, vanag2009cross, de1986reaction}. RD equations are semi-linear parabolic  partial differential equations (PDEs) of the form 
\begin{equation}\label{eq: classical RD}
    \frac{\partial u}{\partial t}(x,t) = Lu(x,t) + \Phi(u(x,t)) \,,
\end{equation}
where $u=u(x,t)$ is the density or concentration of a substance at time $t$ and spatial location $x$, $L$ is a diffusion operator, and $\Phi$ is a function that incorporates the effect of local reactions. 

RD equations defined on graphs, which we call \textit{graph RD equations}, are used for applications in which there is some heterogeneous connectivity between individual nodes. Such applications include population ecology on population networks and disease spread on social networks~\cite{yuasa1998internal, colizza2007reaction, van2021theory,scalise2016emulating}. 

For a graph with $n$ nodes and an adjacency matrix $\{A_{ij}\}_{i,j=1}^n$, the graph RD equation is a system of equations given by
\begin{equation}\label{eq: graph RD intro}
    \frac{\partial}{\partial t} u_i(t) = \frac{1}{n} \sum_{j=1}^n A_{ij} \big(u_j(t)-u_i(t)\big) + \Phi\big(u_i(t)\big) \,, \;\;\;\; i\in \{1,\dots, n\}\,, 
\end{equation}
    where $u_i(t)$ describes the time-varying concentration of a substance on node $i$ of the graph and $\Phi: \R \ra \R$ describes the effect of local reactions of $u$. Therefore, a graph RD equation on a graph with $n$ nodes is a system of $n$ differential equations, and \eqref{eq: graph RD intro} is an analogue of equation \eqref{eq: classical RD} in which, in addition to being discrete, has a nonlocal diffusion term that is weighted by the graph's edge-weights $\{A_{ij}\}$. 

In the present paper, we consider the processes governed by \eqref{eq: graph RD intro} and its stochastic version, as well as their large graph limit as $n\to \infty$. Graphons are well-suited to the study of dynamical systems on graphs because graphons are spatially-continuous objects. Therefore, they allow us to frame graph problems as functional-analytic problems, and they open up the possibility for more sophisticated problem-solving techniques than those that are available for discrete problems. For example, graphons have been used as kernels that induce nonlocal operators~\cite{el2023nonlocal, bramburger2023pattern, medvedev2014nonlinear, heinze2024graph}. The class of nonlocal operators that we are interested in are operators of the form 
\begin{equation}\label{eq: def nonlocal operator}
    \calL^W(u)(x) = \int_\Omega W(x,y) \big(u(y)-u(x)\big) dy \,,
\end{equation}
where $\Omega$ is a bounded and connected domain in $\R^n$ and $W(x,y): \Omega^2 \ra \R$ is a nonnegative kernel function. The operator $\calL^W$ is nonlocal for kernels $W$ that are supported not only on the diagonal $x=y$ \cite{du1992analysis,du2019nonlocal}, because the value of $\calL^W u(x)$ depends on the difference $u(y) - u(x)$ for values of $x$ and $y$ that are not necessarily close to each other. In contrast, a local operator is an operator whose value at a single point depends only on the value (or derivatives) of a function at a single point (see more discussions on the definitions of local and nonlocal operators in \cite{du2019multiscale}.) 
A graphon RD equation is given by \eqref{eq: classical RD} with $L=\calL^W$. If $\Phi\equiv 0$, then \eqref{eq: graph diffusion equation} gives the graph diffusion equation, and \eqref{eq: classical RD} gives the graphon diffusion equation with $L=\calL^W$; see more detailed discussions in Section \ref{sec: graphon diffusion operator}).

\subsection{Related work and our contributions}\label{sec: contributions}

There have been extensive studies of the graph diffusion equations. For brevity in Figure \ref{fig: heat relations}, we denote the graph diffusion equation corresponding to a graph $W_n$ as $\text{Diffusion}^{W_n}$, and the graphon diffusion equation corresponding to a graphon $W$ as $\text{Diffusion}^W$.  
Medvedev~\cite{medvedev2014nonlinear} showed that if a sequence of graphs $W_n$ converges in $L^2$ norm to a graphon $W$, then the sequence of solutions of $\text{Diffusion}^{W_n}$ converges in $L^2$ norm to the solution of $\text{Diffusion}^W$. We extend this result to include the case where $W_n$ converges in cut norm to $W$. The cut norm is a weaker form of convergence than is $L^2$ convergence; see Section \ref{sec: graphons}. 

Diffusion equations are naturally connected to random processes. Angstmann et al.~\cite{angstmann2013pattern} showed that $\text{Diffusion}^{W_n}$ is the master equation of a node-centric continuous-time random walk (RW) on a graph $W_n$, which we denote by $\text{RW}^{W_n}$. In the present paper, we show that RW$^{W_n}$ converges in probability to the solution of Diffusion$^W$ if $W_n$ converges in cut norm to $W$.

\begin{figure}[H]
    \centering
        \begin{tikzpicture}[auto,
                       > = Stealth, 
           node distance = 22mm and 44mm,
              box/.style = {draw=gray, very thick,
                            minimum height=11mm, text width=22mm, 
                            align=center},
       every edge/.style = {draw, <->, very thick},
every edge quotes/.style = {font=\footnotesize, align=center, inner sep=1pt}]

    \node (n11) [box] {$\text{Diffusion}^{W_n}$};
    
    \node (n21) [box, xshift=8cm, yshift=0cm] {$\text{Diffusion}^{W}$};
    
    \node (n31) [box, yshift=0cm, yshift=0cm, below=of n11] {$\text{RW}^{W_n}$};
    
    \draw   [->] (n11) --node[midway,above]{$W_n \lra W$: Medvedev~\cite{medvedev2014nonlinear}}    (n21);
    
    \draw   [->] (n11) --node[midway,below]{$W_n \sqra W$: Theorem \ref{thm: diffusion convergence}}    (n21);
    
    \draw [->] (n31) --node[midway,left]{Master equation~\cite{angstmann2013pattern}}     (n11); 

    \draw[->] (n31) --node[sloped, midway,below,]{Theorem \ref{thm: lln in l2}} (n21);
    \end{tikzpicture}
    \caption{Existing results for graph diffusion equation}
\label{fig: heat relations}
\end{figure}
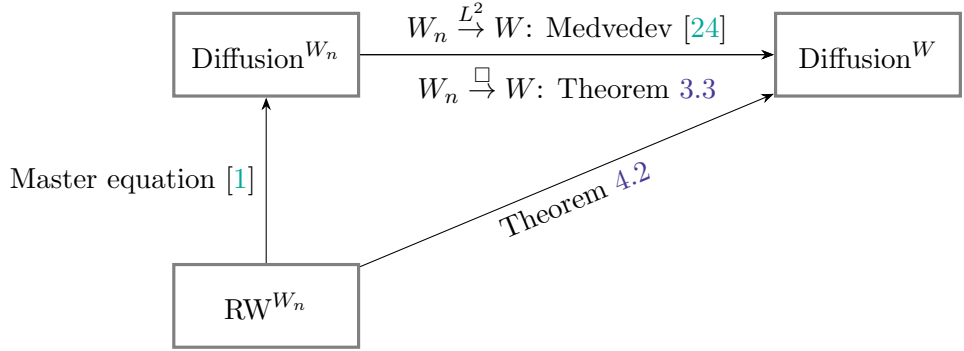

Going beyond pure diffusion equations, 
Figure \ref{fig: rd relations} shows the analogue of
Figure \ref{fig: heat relations} for
more general RD equations, and the graph random walks are generalized to graph random walks that are superimposed with birth-death processes (we denote these by RWBD$^{W_n}$). Angstmann~\cite{angstmann2013pattern} showed that a RD equation on a graph is the master equation of $\text{RWBD}^{W_n}$ on that graph. The reaction function $\Phi$ of the limiting RD equation is the birth rate minus the death rate of the stochastic process. 
Watanabe \cite[Theorem 5.1]{watanabe2022continuum} proves that $\text{RWBD}^{W_n} \ra \text{RD}^{W}$ in probability when $W_n$ is a sequence of ``quotient graphs", which is a special sequence of graphs that converge in the $L^2$ norm to a given graphon. We extend their result to the case in which $W_n$ is any sequence of graphs such that $W_n \sqra W$.

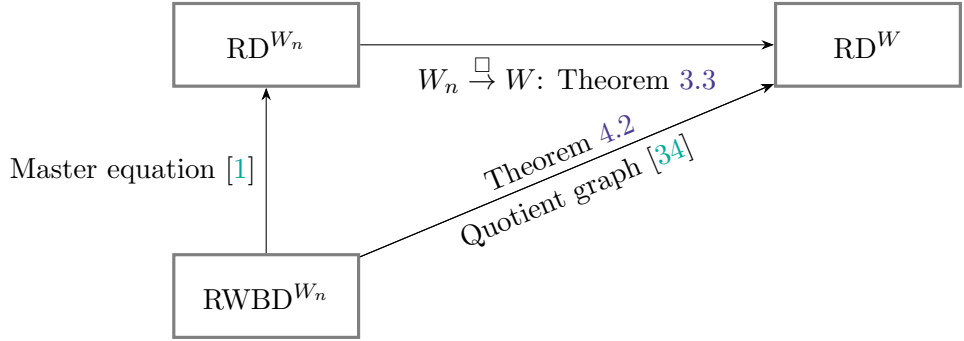
\begin{figure}[H]
    \centering
        \begin{tikzpicture}[auto,
                       > = Stealth, 
           node distance = 22mm and 44mm,
              box/.style = {draw=gray, very thick,
                            minimum height=11mm, text width=22mm, 
                            align=center},
       every edge/.style = {draw, <->, very thick},
every edge quotes/.style = {font=\footnotesize, align=center, inner sep=1pt}]
    \node (n11) [box] {$\text{RD}^{W_n}$};
    \node (n21) [box, xshift=8cm, yshift=0cm] {$\text{RD}^{W}$};
    \node (n31) [box, yshift=0cm, yshift=0cm, below=of n11] {$\text{RWBD}^{W_n}$};
    \draw   [->] (n11) --node[midway,below]{$W_n \sqra W$: Theorem \ref{thm: diffusion convergence}}    (n21);
    \draw [->] (n31) --node[midway,left]{Master equation~\cite{angstmann2013pattern}}     (n11); 
    \draw [->] (n31) --node[sloped, midway, above]{Theorem \ref{thm: lln in l2}}(n21);
    \draw [->] (n31) --node[sloped, midway, below,]{Quotient graph \cite{watanabe2022continuum}}(n21);
    \end{tikzpicture}
    \caption{Reaction--diffusion analogues} 
\label{fig: rd relations}
\end{figure}

A more statistical mechanics--based perspective of this problem is seen in \cite{funaki2019motion} in which the RWBD dynamics are called Glauber--Kawasaki dynamics. The density of these dynamics, with a special choice of birth- and death-rates, are shown to converge to the solution of the Allen--Cahn equation, which is a special case of RD equations. 

In a previous paper~\cite{zhang2024ginzburg}, we proved that sequences of graph Ginzburg--Landau (GL) functionals $\Gamma$-converge to a limiting graphon GL functional. \cite{zhang2024ginzburg} considers a time-independent (i.e., steady-state) problem, whereas this paper considers time-dependent (i.e., dynamical) problems that correspond to gradient flows of energy functionals of which the GL functional is a special case. 

The convergence proofs in~\cite{zhang2024ginzburg} suggest that graphon GL functionals can serve as proxies for graph GL functionals (as well as minimum-cut functionals), whereas the convergence proofs in this paper suggest that graphon RD equations can serve as proxies for graph RD equations. The gradient flow (with respect to the $L^2$ metric) of the graph Ginzburg--Landau functional is the graph Allen--Cahn equation, which is an example of a graph RD equation~\cite{luo2017convergence}. We consider more general RD equations, and our results suggest that graphon RD equations can serve as proxies for graph RD equations. See \cite{serfaty2011gamma} for a method to prove directly that $\Gamma$-convergence of functionals implies convergence of their gradient flows. Our results are more general because 1) we provide convergence rates and 2) we provide convergence in a variety of norms.

Our contributions center on three types of dynamics. These are the linear graph/graphon diffusion equation, the nonlinear graph/graphon reaction--diffusion (RD) equation, and interacting particle system on graphs/graphons.  

Our contributions relating to the graphon diffusion equation are as follows. Lemma \ref{lem: stability of semigroup} gives a maximum bound principle (in $L^p$ norm, for $p \in [1,\infty]$) for the solution of the graph(on) diffusion equation; it implies Theorem \ref{thm:WellDefinedSemigroup} which gives well-posedness of the graph(on) diffusion equation. Theorem \ref{thm: diffusion convergence} proves convergence of the solution of the graph diffusion equation to the solution of the graphon diffusion equation, and it extends \cite{medvedev2014nonlinear} because we assume that $W_n$ and $W$ are $L^1$ graphons and $W_n \sqra W$, whereas \cite{medvedev2014nonlinear} assumes $W_n, W$ are $L^2$ and converge in $L^2$ norm. Theorem \ref{thm: L infty diffusion convergence} obtains convergence of the solution of the graph RD equation to the solution of the graphon RD equation in $L^\infty$ norm provided that $W_n, W$ are $L^\infty$ graphons and $W_n \ra W$ in $L^\infty$ norm. 

Our contributions on the solutions of the graphon RD equation are as follows. Theorems \ref{thm:RdWellPosed-infty} and \ref{thm:RdWellPosed-p} prove, (under suitable assumptions), well-posedness of the strong and mild (respectively) solutions to the graphon RD equation. Theorem \ref{thm:MaxPrinciple} provides a maximum bound principle for the solution of the graphon RD equation. Theorem \ref{thm: graph rd convergence to graphon rd} shows that the solutions of a sequence of graph RD equations converge in $L^p$ norm to the solution of a limiting graphon RD equation. Theorem \ref{thm: graph rd convergence to graphon rd Linfty} shows a similar result in $L^\infty$ norm. These results are novel. 

We also prove a law of large numbers (LLN) result on a stochastic process that converges in probability to the solution of the graphon RD equation. Theorem \ref{thm: lln in l2} is the main result, and it differs from \cite[Theorem 5.1]{watanabe2022continuum} in the choice of norms used for the graph convergence and for the LLN convergence. \cite[Theorem 5.1]{watanabe2022continuum} uses graphons $W$ that are reproducing kernels that generate a reproducing kernel Hilbert space, and graphs that are obtained from $W$ by discretization (and thus converge to $W$ in $L^2$ norm). The sequences of graphs that we consider are more general than those used in \cite{watanabe2022continuum}, namely, we consider any sequences of graphs that converge in cut norm (which we define in section \ref{sec: graphons}) to a limiting graphon. For this more general limit, we are still able to show a weaker $L^2$ convergence, in contrast to the convergence in the RKHS space %now
established for more restrictive cases in 
\cite[Theorem 5.1]{watanabe2022continuum}.

This paper is structured as follows: we first state the assumptions on the model equations and initial conditions in Section \ref{sec: assump}. In Section \ref{sec: graphons}, we introduce definitions and notations relating to graphons and the graphon Laplacian (a.k.a. graphon diffusion operator). In Section \ref{sec: RD equations}, we prove convergence results relating to the graphon diffusion equation and graphon RD equations. In Section \ref{sec: IPS}, we introduce a stochastic particle process on graphs and prove a law of large numbers result that relates the stochastic process to the solution of a graphon RD equation. In Section \ref{sec: conclusion} we conclude with some ideas for future work.

\subsection{Assumptions}
\label{sec: assump}

We present some assumptions on the kernel $W$, the reaction term $\Phi$, and the initial conditions. We also list the theorems for which these assumptions are needed. 

We first consider assumptions on graphs and graphons. Note that we treat any graph as a piecewise constant graphon (see Section \ref{sec: graphons and cut norm}), so it suffices to state assumption on graphons. We impose \Cref{assump:W} on all graphs and graphons throughout the rest of the paper. 
\begin{assumption}\label{assump:W}
Let $W$ be a graphon. We assume that 
    \begin{equation}\label{eq:assump:W}
        \begin{gathered}
        W \in L^1((0,1)^2)\,, 
        \qquad  W(x,y) = W(y,x)  \text{ and } W(x,y) \geq 0\,,\:\: \forall (x,y) \in (0,1)^2\,, 
        \\
        \text{and } d_W(x) := \int_0^1 W(x,y) dy \leq 1 \quad \text{ for a.e. } x \in (0,1)\,.
        \end{gathered}
    \end{equation}
\end{assumption}
\Cref{assump:W2} states additional assumptions on $W$ and $W_n$ needed for various results. 
\begin{assumption}\label{assump:W2}
Two different sets of assumptions that we use for sequences of graphs $W_n$ converging to a graphon $W$ are 
 \begin{enumerate}[i)] 
    \item $W_n, W \in L^\infty((0,1)^2)$ and $W_n \ra W$ in $L^\infty$. These are needed in Theorems \ref{thm: L infty diffusion convergence}, and \ref{thm: graph rd convergence to graphon rd Linfty}. 
    \item $W_n, W \in L^1((0,1)^2)$ and $W_n \sqra W$. These are needed for Theorems \ref{thm: diffusion convergence}, \ref{thm: graph rd convergence to graphon rd}, and \ref{thm: lln in l2}. 
\end{enumerate}
\end{assumption}
\begin{assumption}\label{assump: phi} %
Four different assumption sets 
on $\Phi: \R \ra \R$ are
\begin{enumerate}  
    \item[i)]  $\Phi(0)= 0$, and $\Phi$ is locally Lipschitz continuous (that is, $\Phi$ is Lipschitz continuous on every bounded subset of $\R$).  These are needed for 
    Theorems \ref{thm:RdWellPosed-infty}, \ref{thm:MaxPrinciple}, and \ref{thm: graph rd convergence to graphon rd Linfty}. 
    \item[ii)] $\Phi(0)= 0$, $\Phi$ is locally Lipschitz continuous, and there exist constants $M_1$, $M_2$ with $M_1 < M_2$ such that $\Phi(M_2) \leq 0 \leq \Phi(M_1)$. These are needed for Theorem \ref{thm:MaxPrinciple}. 
    \item[iii)] $\Phi$ satisfies $\Phi(0)= 0$, and is uniformly Lipschitz continuous; that is, $\Phi$ is Lipschitz continuous with constant $K$ on all of $\R$. These are assumed in Theorems \ref{thm:RdWellPosed-p}, \ref{thm: graph rd convergence to graphon rd}.
    \item[iv)] $\Phi(x) = b(x) - d(x)$, where $b, d: \R \ra \R$ are uniformly Lipschitz-continuous and satisfy $b(0) = d(0) = 0$. These are needed in \Cref{thm: lln in l2}, where $b$ and $d$ are birth and death rates, respectively, of a stochastic particle system. 
\end{enumerate}
\end{assumption}

For all the graph and graph dynamics under consideration, initial conditions for the solution at $t=0$, denoted by $u(0)$ for the graphon equations, and $u_n(0)$ for the graph case, must be imposed.
Assumptions on $u_n(0)$ and $u(0)$ are as follows:
\begin{assumption}\label{assump: u0} The initial conditions satisfy, respectively, the following conditions for the corresponding results.
\begin{enumerate}[i)]
    \item $u(\cdot, 0) \in L^\infty((0,1))$ is needed for
    Theorems \ref{thm:RdWellPosed-infty}, \ref{thm:MaxPrinciple}, and \ref{thm: lln in l2}
    \item $u(\cdot, 0) \in L^p((0,1))$ is needed for
    Theorems \ref{thm:WellDefinedSemigroup} and \ref{thm:RdWellPosed-p}
    \item $u_n(\cdot, 0) \in L^p((0,1)), u(\cdot, 0) \in L^\infty((0,1))$ are needed for
    Theorems \ref{thm: diffusion convergence}, \ref{thm: L infty diffusion convergence}, and \ref{thm: graph rd convergence to graphon rd}
    \item $u_n(\cdot, 0) \in L^\infty((0,1)), u(\cdot, 0) \in L^\infty((0,1))$ are needed for Theorem \ref{thm: graph rd convergence to graphon rd Linfty}.
\end{enumerate}
\end{assumption}

%%%%%%%%%%%%%%%%%%%%%%%%%%%%%%%%%
\section{Graphons and norms of convergence}\label{sec: graphons}

For $n\in \N$, a graph on $n$ nodes is called $W_n$ and its nodes are denoted by the set $[n] = \{1,2,\dots, n\}$. Define the intervals 
\begin{equation}\label{eq: def I_k}
\begin{gathered}
    I_k = ((k-1)/n, \, k/n] \text{ for } k=1, \dots, n-1\,,
    \\
    I_n = ((n-1)/n, \, 1)\,,
    \end{gathered}
\end{equation} 
so that $\cup_{k=1}^n I_k = (0,1)$.   

\subsection{Graphons and cut norm}\label{sec: graphons and cut norm}
A graphon is a bounded, measurable, and symmetric function $W: (0,1)^2 \ra \R$ that serves as a notion of a large-graph limit~\cite{lovasz2012large}. Graphons are often defined as functions on the closed interval $[0,1]^2$, but we use the open interval because the open interval is better-suited to functional-analytic tools. It does not matter whether graphons are defined on $[0,1]^2$ or $(0,1)^2$, because graphons are Lebesgue-measurable. More generally, graphons can be defined as functions $W: \Omega^2 \ra \R$ where $\Omega \in \R^n$ is any connected and bounded domain.

Graphs are special cases of graphons. We consider graphs that are weighted, undirected, and simple (i.e., there are no self-edges or multi-edges). A graph $W_n$ with nodes $[n]$ has an associated adjacency matrix $A^{(n)}$ with entries $A^{(n)}_{ij}$ for $i,j \in [n]$. The graph $W_n$ is a piecewise-constant graphon that takes the constant value $A^{(n)}_{ij}$ on each product interval $I_i \times I_j$ for $i,j \in [n]$, where $I_i$ and $I_j$ are defined as in \eqref{eq: def I_k}. Thus, the graph $W_n$ can be expressed as the piecewise-constant graphon 
\begin{equation}   \label{eq: step-graphon} 
    W_n(x,y) = \left\{ A^{(n)}_{ij} \;\; \text{ for } (x,y) \in I_i \times I_j\,, \,\,  i,j \in \{1,\ldots, n\}  \right\}  \,.
\end{equation}
In this way, one can identify each finite graph $W_n$ with a graphon, which we will also denote by $W_n$ (see Remark \ref{remark: identify W_n}). As an example of this relation between the adjacency matrix and the graphon representation of a graph, Figure \ref{fig: example graphon} depicts the 4-cycle graph, its associated adjacency matrix, and its corresponding raphon (where white represents the value 0 and black represents the value 1). Figure \ref{fig: example graphon} is reproduced from our previous paper \cite{zhang2024ginzburg}. 

\begin{figure}[H]
\captionsetup[subfigure]{font=footnotesize}
\centering
    \subcaptionbox{The 4-cycle graph. }[.32\textwidth]{
        \begin{tikzpicture}
        \node at (0,0)[circle,fill,inner sep=2pt]{};
        \node at (0,2)[circle,fill,inner sep=2pt]{};
        \node at (2,0)[circle,fill,inner sep=2pt]{};
        \node at (2,2)[circle,fill,inner sep=2pt]{};

        \draw (0,2) node[anchor=south east] {$1$};
        \draw (2,2) node[anchor=south west] {$2$};
        \draw (2,0) node[anchor=north west] {$3$};
        \draw (0,0) node[anchor=north east] {$4$};
        \draw[-,very thick] (0,0)--(0,2);
        \draw[-,very thick] (0,0)--(2,0);
        \draw[-,very thick] (0,2)--(2,2);
        \draw[-,very thick] (2,0)--(2,2);
    \end{tikzpicture}
    }
    \subcaptionbox{Its associated adjacency \\ matrix.}[.32\textwidth]{
    {\large $\begin{pmatrix}
        0&1&0&1
        \\
        1&0&1&0
        \\
        0&1&0&1
        \\
        1&0&1&0
    \end{pmatrix}$}
    }
    \hspace{.2cm}
    \subcaptionbox{Its corresponding graphon.}[.32\textwidth]{
    \resizebox{100pt}{100pt}{%
        \begin{tikzpicture}
            %grid and axes
            \draw[help lines, color=black, solid, thick, step = 1] (0,0) grid (4,4);
            
            %labels
            \foreach \x in {1,..., 4}
                \node at (\x-.5, 4+.4) {\x} node at (0-.5, 4-\x+.5) {\x};
                
            %filled squares
            \draw[color=black, fill = black] (1,3) rectangle (2,4); 
            \draw[color=black, fill = black] (0,2) rectangle (1,3); 
            \draw[color=black, fill = black] (3,3) rectangle (4,4); 
            \draw[color=black, fill = black] (2,2) rectangle (3,3); 
            \draw[color=black, fill = black] (1,1) rectangle (2,2); 
            \draw[color=black, fill = black] (3,1) rectangle (4,2); 
            \draw[color=black, fill = black] (0,0) rectangle (1,1); 
            \draw[color=black, fill = black] (2,0) rectangle (3,1);  
\end{tikzpicture}}}
\caption{An example of a graph and its associated adjacency matrix and step-graphon.}
\label{fig: example graphon}
\end{figure}
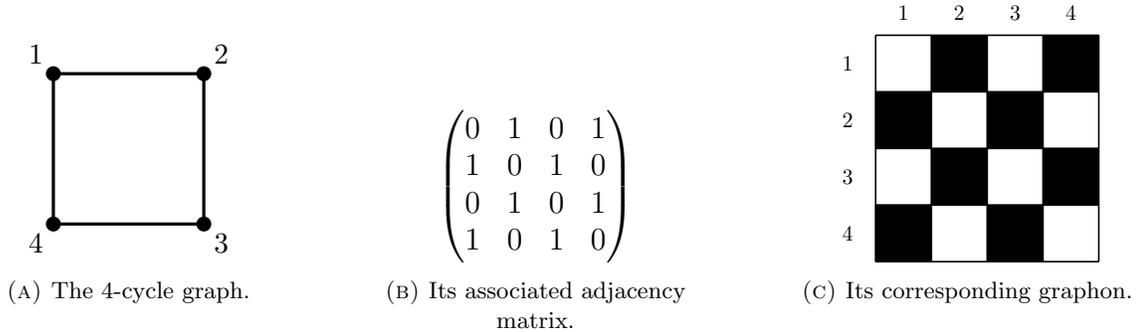

In other words, we associate the $i$th node with the interval $I_i$; the value of the graphon on the interval $I_i \times I_j$ gives the weight of the edge between nodes $i$ and $j$. In the $n\ra \infty$ limit, the intervals shrink and $(0,1)$ can be viewed as a infinite node-set and $W(x,y)$ can be viewed as the weight of the edge between the infinitesimal nodes $x$ and $y$. 

In this paper, we apply a functional-analytic viewpoint to graphons, and we consider graphons that reside in different function spaces. The original graphons defined in~\cite{lovasz2012large} are bounded functions $W \in L^\infty((0,1)^2)$, and they are limits of dense sequences of graphs. More generally, graphons $W \in L^p((0,1)^2)$ include limits of sparse sequences of graphs~\cite{bollobas2007metrics,borgs2018LpII}. In this paper, we think of graphons simply as $L^p$ functions, for $p \in [1,\infty]$, rather than as limits of dense or sparse sequences of graphs. $L^p$ graphons can be associated with kernels for integral operators from $L^q((0,1))$ to $L^p((0,1))$ for $q = \frac{p}{p-1}$ (c.f. H\"older's inequality). %~\cite{levermore2009integral}
Similarly, $L^\infty$ graphons induce integral operators from $L^1((0,1))$ to $L^\infty((0,1))$.  In Section \ref{sec: RD equations}, we assume that $W_n \in L^1((0,1)^2)$, and in Section \ref{sec: IPS}, we assume that $W_n \in L^\infty((0,1)^2)$.

Graphs and graphons converge to a limiting graphon with respect to the \textit{cut norm}~\cite{lovasz2012large}. The cut norm induces a metric such that the set of $L^p$ graphons is a compact metric space~\cite{lovasz2007szemeredi, borgs2019LpI} for $p\in (1,\infty]$. The set of $L^1$ graphons is compact under the additional assumption that $W_n$ are uniformly integrable \cite[Theorem C.7]{borgs2019LpI}. Therefore, any bounded sequence $\{W_n\}_{n\in \N}$ of graphons has a subsequence $W_{n'}$ such that $\|W_{n'} - W\|_\square \ra 0$. 

\begin{definition}
    The \emph{cut norm} of a graphon $W$ is defined as 
\begin{equation} 
    \|W\|_\square = \sup_{S\subseteq (0,1)} \int_{S \times S^c} W(x,y)\,dx\,dy \,. 
\label{eq: cut norm} 
\end{equation}
\end{definition}
There are other equivalent definitions of the cut norm \eqref{eq: cut norm}~\cite{janson2010graphons}. One version that we use in this paper is 
\begin{equation}
    \| W \|_\square = \sup_{\phi,\psi \in L^\infty((0,1);[-1,1])} \int_0^1 \int_0^1 W(x,y) \phi(x)\psi(y)\,dx\,dy \,,
\label{eq: cut norm 1}
\end{equation}
where $L^\infty((0,1);[-1,1])$ is the set of functions that map $(0,1)$ to $[-1,1]$.

\begin{remark}\label{remark: identify W_n} Any finite weighted graph $W_n: [n] \ra \R$ is at cut norm distance $0$ from its corresponding step-graphon $W_n$. Therefore, we use the notation $W_n$ for both objects. 
\end{remark}

\begin{remark} 
For $p\in [1,\infty]$, we can approximate any $L^p$ graphon arbitrarily closely in cut norm by a finite graph. Conversely, the set of $L^p$ graphons for $p \in (1,\infty]$ is compact with respect to cut norm. The set of $L^1$ graphons that are uniformly integrable is compact with respect to cut norm. See ~\cite[Remark 4.6]{janson2010graphons}, \cite[Theorem 2.13]{borgs2019LpI}, \cite[Theorem C.7]{borgs2019LpI}. 
\end{remark}

\begin{remark}
    The cut norm is equivalent to the operator norm of the kernel operator that is induced by the graphon $T_W(f) = \int_0^1 W(x,y) f(y) \, d\mu(y)$, which is a linear operator $T_W: L^\infty((0,1)) \ra L^1((0,1))$. In fact, the terms ``graphon" and ``kernel" are sometimes used interchangeably~\cite{janson2010graphons}. 
\end{remark}

Because the present paper extends existing results (see Section \ref{sec: contributions}), in part by altering the norm of convergence, it is important to discuss the different norms. From \cite[Remark 4.1]{janson2010graphons}, we have that 
\begin{equation}\label{eq: norm inequalities}
    \left| \int_0^1 \int_0^1  W(x,y) \, dx \, dy \right| \leq \|W\|_\square \leq \|W\|_1 \leq \|W\|_2 \,. 
\end{equation}
This means that convergence in cut norm is weaker than convergence in $L^1$ norm, which is weaker than convergence in $L^2$ norm, in the sense that convergence in $L^2$ norm implies convergence in $L^1$ norm which implies convergence in cut norm.

We also define the degree of a graphon. Recall that the degree of node $i$ of a graph that has adjacency matrix $A$ is 
\[ \text{deg}_A(i) = \sum_{j=1}^n A_{ij} \,.\]
\begin{definition}
    The \textit{degree} of a graphon $W$ at node $x$ is defined as 
    \begin{equation}\label{eq: def degree}
        d_W(x) = \int_0^1 W(x,y) dy \,. 
    \end{equation}
\end{definition}

\subsection{Graphon diffusion operator}\label{sec: graphon diffusion operator}

Let $W \in L^\infty((0,1)^2)$ be a graphon. The corresponding graphon diffusion operator (a.k.a. the graphon Laplacian) is defined as 
\begin{equation} \label{eq: def graphon diffusion operator}
    \calL^W u(x, t) = \int_0^1 W(x,y) \big(u(y, t)-u(x, t)\big) \, dy 
\end{equation}
for $u \in L^1((0,1))$.  Consider a graph with nodes $[n]$ that has adjacency matrix $A^{(n)}_{ij}$ and let $u_n: [n] \ra \R$ be a function on the nodes of the graph. The graph diffusion operator (a.k.a. the graph Laplacian) is defined as \cite{merris1994laplacian}
\begin{equation}\label{eq: def graph diffusion operator}
    \calL^{W_n} u_n(t,i) = \frac{1}{n}\sum_{j=1}^{n} A^{(n)}_{ij} \big(u_n(j,t) - u_n(i,t)\big) \,.
\end{equation} 
By defining the step graphon $W_n$ that corresponds to the adjacency matrix $A^{(n)}_{ij}$ using the relation \eqref{eq: step-graphon}, we express the graph diffusion operator as the integral operator 
\begin{equation}\label{eq: def graph diffusion integral operator}
    \calL^{W_n} u(x, t) = \int_0^1 W_n (x,y) \big(u(y, t) - u(x, t)\big) \, dy 
\end{equation}
where $u$ is a step-function on $(0,1)$ that is equal to $u_n(k,t)$ on each $I_k$ for $k= 1,\dots, n$. Furthermore, $\calL^{W_n}$ can act on any function in $L^1((0,1))$. Therefore, \eqref{eq: def graph diffusion integral operator} is a special case of the graphon diffusion operator \eqref{eq: def graphon diffusion operator}. 

For $t>0$ and $x\in (0,1)$, the graphon diffusion equation is defined as 
\begin{equation}\label{eq: graphon diffusion equation}
    \frac{\partial u}{\partial t}(x,t) = \calL^W u(x,t) \,. 
\end{equation}
The graph diffusion equation 
is similarly defined as 
\begin{equation}\label{eq: graph diffusion equation}
        \frac{\partial u_n}{\partial t}(x,t) = \calL^{W_n} u_n(x,t) \,. 
\end{equation}
Note that $\int_0^1 \calL^W u(x,t) dx = 0$ due to the symmetry of $W$. As a result, $\int_0^1 \frac{\partial u}{\partial t}(x,t) dx = 0$, which means that the solutions to the graph and graphon diffusion equations conserve mass. This can be interpreted as a nonlocal version of the mass conservation property associated with the classical local diffusion equation with a local
zero-flux (Neumann) boundary condition.

\section{Reaction--diffusion equations on graphs and graphons} \label{sec: RD equations}

In this section, we introduce graph and graphon diffusion equations and reaction--diffusion (RD) equations. Section \ref{sec: graphon diffusion} provides a contraction result for the solution of the graphon diffusion equation and establishes that the solutions of a sequence of graph diffusion equations converges to the solution of a graphon diffusion equation, under certain assumptions on the graphons and norms of convergence. 

\subsection{Graphon diffusion equation} \label{sec: graphon diffusion}

We first analyze the diffusion operator $\calL^W$. For any $p \in [1,\infty)$, and for any $u \in L^p((0,1))$, we have 
\begin{align*}
        \| \cL^W u \|_p^p  
        &= \int_0^1 \left(\int_0^1 W(x,y) (u(y)-u(x)) dy\right)^p dx
        \\
        &\leq 2^{p-1} \int_0^1 \left| \int_0^1 W(x,y) u(y) \, dy \right|^p + \left| \int_0^1 W(x,y) u(x) \, dy \right|^p \, dx 
        \\
        &\leq 2^{p-1} \int_0^1 \left( \int_0^1 W(x,z) \, dz \right)^{p-1} \int_0^1 W(x,y) |u(y)|^p \, dy dx 
        \\
        &\qquad + 2^{p-1} \int_0^1 \left( \int_0^1 W(x,y) \, dy \right)^{p} |u(x)|^p \, dx 
        \\
        &= 2^{p-1} \int_0^1 \left( \int_0^1 d_W(x)^{p-1} W(x,y) \, dx \right) |u(y)|^p  dy + 2^{p-1} \int_0^1 d_W(x)^p |u(x)|^p \, dx 
        \\
        &\leq 2^p \| u \|_p^p \,. \numberthis \label{eq:DiffusionOperatorBd}
\end{align*}
The first inequality above follows from the fact that $\left(\frac{a+b}{2}\right)^p \leq \frac{a^p+b^p}{2}$ for $p>1$ because the function $x^p$ is convex in $x$. The second inequality is obtained by rewriting the quantity $W(x,y) u(y)$ as $W(x,y)^{1/p} u(y) W(x,y)^{(p-1)/p}$ and applying H\"older's inequality. The second equality uses the definition of the degree function of graphons. The last equality uses the definition of the degree of a node of a graphon (see equation \eqref{eq: def degree}).

We similarly obtain 
\begin{equation}\label{eq:DiffusionOperatorBd2}
    \| \cL^W u \|_\infty \leq 2 \| u \|_\infty, \qquad \forall u \in L^\infty((0,1)).
\end{equation}
Therefore, $\cL^W$ is a bounded linear operator from $L^p((0,1))$ to $L^p((0,1))$ for all $p \in [1,\infty]$.
As a consequence, 
$\cL^W$ generates a uniformly continuous semigroup on $L^p((0,1))$ for each $p \in [1,\infty]$~\cite[Section 1.1]{pazy2012semigroups}. We denote this semigroup by $\{ e^{t\cL^W } \}_{t \geq 0}$. 

We use \eqref{eq:DiffusionOperatorBd}--\eqref{eq:DiffusionOperatorBd2} in \Cref{lem: stability of semigroup}, which implies that solutions to the graphon diffusion equation \eqref{eq: graphon diffusion equation} satisfy a stability property. When proving \Cref{lem: stability of semigroup}, We use the fact that 
\begin{equation}\label{eq: time deriv of lp norm}
    \frac{\partial}{\partial t} \|u(\cdot,t)\|^p_p = p \left \langle \left|u(\cdot,t) \right|^{p-2} u(\cdot,t), \, \frac{\partial}{\partial t} u(\cdot,t) \right \rangle 
\end{equation}
which can be obtained via chain rule.

\begin{lemma} \label{lem: stability of semigroup}
    Let $p \in [1,\infty]$ and $T > 0$. Let $u \in C^1([0,T]; L^p((0,1)))$ be a classical solution to the graphon diffusion equation \eqref{eq: graphon diffusion equation} with the initial conditions $u_0\in L^p((0,1))$. Assume that the graphon $W\in L^1((0,1)^2)$ satisfies Assumption \ref{assump:W}. Then,
    \begin{equation*}
        \sup_{t \in [0,T] } \left\| u(\cdot,t)
       \right\|_p \leq \|u_0 \|_p. 
    \end{equation*}
\end{lemma}

\begin{proof}  
First let $p \in [1,\infty)$.  With $u(\cdot,t) \in L^p((0,1))$, we also have $\cL^W u(\cdot,t)$. Then we have 
\begin{align*}
    &\frac{\partial}{\partial t} \|u(\cdot,t)\|^p_p \\
    =& p \left \langle \left|u(\cdot,t) \right|^{p-2} u(\cdot,t), \, \frac{\partial}{\partial t} u(\cdot,t) \right \rangle 
    \\
    =& p \int_0^1 \int_0^1 W(x,y) \left[ u(y,t) - u(x,t) \right] \left|u(x,t) \right|^{p-2} u(x,t) \, dy \, dx 
    \\
    =& \frac{p}{2} \int_0^1  \int_0^1 W(x,y) \left[ u(y,t) - u(x,t) \right] \left|u(x,t) \right|^{p-2} u(x,t) \, dy \, dx  
    \\
    &\qquad + \frac{p}{2} \int_0^1 \int_0^1 W(y,x) \left[ u(x,t) - u(y,t) \right] \left|u(y,t) \right|^{p-2} u(y,t)\, dy \, dx 
    \\
    =& - \frac{p}{2} \int_0^1  \int_0^1 W(x,y) \left[ u(y,t) - u(x,t) \right] \left[ \left|u(y,t) \right|^{p-2} u(x,t) - \left|u(x,t) \right|^{p-2} u(x,t) \right] \, dy \, dx 
    \\
    \leq& 0 \,.
\end{align*} 
The last inequality is due to the fact that the function $a \mapsto |a|^{p-2} a$ is monotone, which implies that the quantities $ \left[ u(y,t) - u(x,t) \right]$ and $\left[ \left|u(y,t) \right|^{p-2} u(x,t) - \left|u(x,t) \right|^{p-2} u(x,t) \right]$ have the same sign. 

Because we have shown that $\|u(\cdot,t)\|^p_p $ is decreasing as a function of $t$, we have that 
\begin{equation}\label{eq: boundedness of semigroup}
    \left\|u(\cdot,t) \right\|_p \leq \left\| u_0 \right\|_p \,.
\end{equation}
This completes the theorem when $p < \infty$. The $p = \infty$ case is obtained by taking the inequality \eqref{eq: boundedness of semigroup} to the $p \ra \infty$ limit. 
\end{proof} 

A consequence of \Cref{lem: stability of semigroup} is that solutions of the graphon diffusion equation are unique. Existence of the solutions follows from the existence of the semigroup $\{e^{t\calL^W}\}_{t\geq 0}$. Therefore, we conclude the following well-posedness result. 

\begin{theorem}\label{thm:WellDefinedSemigroup}
    Let $T > 0$ and $p \in [1,\infty]$. Let $u_0 \in L^p((0,1))$. Then there exists a unique solution $u \in C^1([0,T];L^p((0,1)))$ to the graphon diffusion equation \eqref{eq: graphon diffusion equation} with initial condition $u(x,0) = u_0(x)$.
    The solution is given by $u(\cdot,t) = e^{t\cL^W } u_0$ for all $t \geq 0$.
    Moreover, 
    we have the contraction estimate
\begin{equation}\label{eq:Semigroup:LpStab}
        \big\| e^{t\cL^W} u_0 
        \big\|_p \leq \| u_0 
        \|_p \qquad \forall t \geq 0 \,.
    \end{equation}
\end{theorem}

%Theorem 1.2 of chapter 1 of Pazy semigroups says that a linear operator $A$ is the infinitesimal generator of a uniformly continuous semigroup iff $A$ is a bounded linear operator. Our $A$ is $\calL^{W}$ and it is linear and bounded, hence it is the infinitesimal generator of a uniformly continuous semigroup $e^{\calL^W t}$. 

%The following lemma shows the stability property of the semigroup. This property is also called nonexpansiveness, and in the $p=\infty$ case, it is also called the maximum property. 

Next, we consider the convergence of semigroup operators given the convergence of their defining graph kernels. Theorem \ref{thm: diffusion convergence} extends \cite[Theorem 4.1]{medvedev2014nonlinear} to the case of cut norm convergence. Specifically, \cite[Theorem 4.1]{medvedev2014nonlinear} proves that the solutions of the graph diffusion equations converge to the solution of a limiting graphon diffusion equation in the case when the graphs converge in $L^2$ norm to the limiting graphon. We prove a similar result but assume that the graphs converge in cut norm to the limiting graphon, which is a weaker assumption (see equation \eqref{eq: norm inequalities}). 

\begin{theorem} \label{thm: diffusion convergence}
Let $p \in [1,\infty)$ and let $T > 0$. Assume that $W$ and $W_n$ satisfy \Cref{assump:W2}ii) and let $u_n \in C^1([0,T];L^p((0,1)))$ be the solution to the graph diffusion equation \eqref{eq: graph diffusion equation} and let $u \in C^1([0,T];L^\infty((0,1)))$ be the solution to the graphon diffusion equation \eqref{eq: graphon diffusion equation}; let their initial conditions satisfy \Cref{assump: u0}iii).
Then, 
\begin{equation} \label{eq: diffusion convergence bound}
    \|u(\cdot, t) - u_n(\cdot, t)\|_p \leq \|u(\cdot, 0) - u_n(\cdot, 0)\|_p + 4t \|u(\cdot,0)\|_\infty \|W_n-W\|_\square^{1/p} \,.
\end{equation}
\end{theorem}

\begin{proof}
Define the error $e_n(x,t) = u_n(x,t) - u(x,t)$. We will show that the $L^p$ norm of the error decreases in $t$. Due to the definitions of the graph and graphon diffusion equation \eqref{eq: graphon diffusion equation}--\eqref{eq: graph diffusion equation}, we have 
\begin{align*}
    \frac{\partial e_n}{\partial t}(x,t) &= \frac{\partial u_n}{\partial t}(x,t) - \frac{\partial u}{\partial t}(x,t)
    \\
    &= \int_0^1 W_n(x,y)\big(u_n(y,t) - u_n(x,t)\big) \, dy - \int_0^1 W(x,y)\big(u(y,t) - u(x,t)\big) \, dy
    \\
    &= \int_0^1 W_n(x,y)\big(e_n(y,t)-e_n(x,t)\big)\, dy + \int_0^1 (W_n-W)(x,y) \big(u(y,t)-u(x,t)\big)\, dy \,.\numberthis \label{eq: partial deriv of e_n}
\end{align*}
We can express \eqref{eq: partial deriv of e_n} as 
\begin{equation}\label{eq: time deriv of e_n}
    \frac{\partial e_n}{\partial t} = \calL^{W_n} e_n + \cL^{W_n-W}u\,.
\end{equation}

Equipped with this fact, we now compute the time-derivative of $ \| e_n(\cdot,t) \|_p$. 
Using the chain rule, we obtain
\begin{equation}\label{eq: chain rule derivative of Lp}
    \frac{\partial}{\partial t} \left[ \| e_n(\cdot,t) \|_p \right] = \frac{\partial}{\partial t} \left[ \left( \| e_n(\cdot,t) \|_p^p \right)^{1/p} \right]
    =  \frac{1}{p} \| e_n(\cdot,t) \|_p^{1-p} \frac{\partial}{\partial t} \left[ \| e_n(\cdot,t) \|_p^p \right]
\end{equation}
because it is more straightforward to take the derivative of $\|e_n(\cdot, t)\|^p_p$ than of $\|e_n(\cdot, t)\|_p$. Define $\varphi(a) = |a|^{p-2} a$ for $a \in \R$. Using equations \eqref{eq: time deriv of lp norm} and \eqref{eq: time deriv of e_n}, we obtain 
\begin{align*}
    \frac{1}{p} \frac{\partial}{\partial t} \|e_n(\cdot, t)\|_p^p 
    &= \frac{1}{p} p \left \langle \varphi(e_n(\cdot,t)), \frac{\partial}{\partial t} e_n(\cdot, t) \right \rangle 
    \\
    &= \int_0^1 \varphi(e_n(x,t))
    \, \calL^{W_n} e_n(x,t) \, dx  + \int_0^1 \varphi(e_n(x,t))
    \, \cL^{W_n-W}u(x,t) \, dx
    \\
    &=: (\mathrm{I}) + (\mathrm{II}) \,. \numberthis \label{eq: (I) + (II)}
\end{align*}
We know that $(\mathrm{I}) \leq 0$, because a calculation similar to that of the proof of \Cref{lem: stability of semigroup} gives
\begin{align*}
    (\mathrm{I})
    =& \int_0^1 \int_0^1 W_n(x,y) \left[ e_n(y,t) - e_n(x,t) \right] 
     \varphi(e_n(x,t) ) \, dy \, dx 
    \\
    =& \frac{1}{2} \int_0^1  \int_0^1 W_n(x,y) \left[ e_n(y,t) - e_n(x,t) \right] 
     \varphi(e_n(x,t)) \, dy \, dx  
    \\
    &\qquad + \frac{1}{2} \int_0^1 \int_0^1 W_n(y,x) \left[ e_n(x,t) - e_n(y,t) \right] \varphi(e_n(y,t))  
    \, dy \, dx 
    \\
    =& - \frac{1}{2} \int_0^1  \int_0^1 W_n(x,y) \left[ e_n(y,t) - e_n(x,t) \right] \left[ 
         \varphi(e_n(y,t))
         - \varphi(e_n(x,t)) \right] \, dy \, dx 
    \\
    \leq& 0. \numberthis \label{eq: bound on (I)}
\end{align*}

Next, we bound (II). 
Let $q = p/(p-1)$. We may first rewrite (II) as
\begin{align*}
    (\mathrm{II}) &= \int_0^1 \varphi(e_n(x,t)) (\cL^{W_n-W} u(x,t)) dx 
    \\
    &= \frac{1}{2} \int_0^1 \int_0^1 (W_n-W)(x,y) (u(y,t)-u(x,t)) \varphi(e_n(x,t)) \, dydx
    \\
    &\qquad + \frac{1}{2} \int_0^1 \int_0^1 (W_n-W)(y,x) (u(x,t)-u(y,t)) \varphi(e_n(y,t)) \, dydx
    \\
    &= - \frac{1}{2} \int_0^1 \int_0^1 (W_n-W)(x,y) (u(y,t)-u(x,t)) \big( \varphi(e_n(y,t))-\varphi(e_n(x,t)) \big) \, dydx.   \numberthis\label{eq: bound II}
    \end{align*}
    Using Holder's inequality, the definition of the cut norm stated in \eqref{eq: cut norm 1}, 
and the fact that $\|u(\cdot, t)\|_\infty \leq \|u(\cdot, 0)\|_\infty$ due to \Cref{lem: stability of semigroup}, 
we have 
    \begin{align*}   
    (\mathrm{II})
    &\leq \frac{1}{2} \left( \int_0^1 \int_0^1 | (W_n-W)(x,y)|\, |u(y,t)-u(x,t)|^p \, dy dx \right)^{1/p}
    \\
    &\qquad \cdot 
    \left( \int_0^1 \int_0^1| (W_n-W)(x,y)| \, |\varphi(e_n(y,t))-\varphi(e_n(x,t))|^q \, dy dx \right)^{1/q}
    \\
    &= \frac{1}{2} \left( \int_0^1 \int_0^1 | (W_n-W)(x,y)|\, |u(y,t)-u(x,t)|^p \, dy dx \right)^{1/p}\cdot (\mathrm{III})
    \\
    &\leq \| u(\cdot,t) \|_\infty \|W_n - W\|_\square^{1/p}\cdot (\mathrm{III}) 
    \\
    &\leq \| u(\cdot,0) \|_\infty \|W_n - W\|_\square^{1/p}\cdot (\mathrm{III}) \,. \numberthis \label{eq: term ii lastline}
\end{align*}
Now we use the triangle inequality and the bound on the degree $d_W(x)$ (from \Cref{assump: u0}iii)) to bound the term $(\mathrm{III})$:
\begin{align*}
        (\mathrm{III}) &\leq \left( \int_0^1 \int_0^1 |(W_n-W)(x,y)| |\varphi(e_n)(y,t)|^q \, dy dx \right)^{1/q}
        \\
        &\quad + \left( \int_0^1 \int_0^1 |(W_n-W)(x,y)| |\varphi(e_n)(x,t)|^q \, dy dx \right)^{1/q}
        \\
        &= \left( \int_0^1 \int_0^1 |(W_n-W)(x,y)| dx |e_n(y,t)|^p \, dy \right)^{(p-1)/p}
        \\
        &\quad + \left( \int_0^1 \int_0^1 |(W_n-W)(x,y)| dy |e_n(x,t)|^p \, dx \right)^{(p-1)/p}
        \\
        &\leq 4 \| e_n(\cdot,t) \|_p^{p-1}. \numberthis \label{eq: term iii lastline}
\end{align*}
Substituting \eqref{eq: bound on (I)}, \eqref{eq: term ii lastline}, and \eqref{eq: term iii lastline} into \eqref{eq: (I) + (II)}, we have 
\begin{equation*}
    \frac{1}{p} \frac{\partial}{\partial t} \|e_n(\cdot, t)\|_p^p \leq 4 \| u(\cdot,0) \|_\infty \|
    W_n - W \|_\square^{1/p} \| e_n(\cdot,t) \|_p^{p-1}\,.  
\end{equation*}
Using the identity \eqref{eq: chain rule derivative of Lp}, this becomes 
\begin{equation}\label{eq: l^p diffusion error bound}
    \frac{\partial}{\partial t}  \|e_n(\cdot,t) \|_p  \leq 4 \| u(\cdot,0) \|_\infty \|
    W_n - W \|_\square^{1/p}.
\end{equation}
Gronwall's inequality applied to \eqref{eq: l^p diffusion error bound} gives the desired bound \eqref{eq: diffusion convergence bound}.
\end{proof}

\begin{remark}
    One can generalize Theorem \ref{thm: diffusion convergence} to nonlinear diffusion, in which the term $u(y,t) - u(x,t)$ in the diffusion operator is replaced by $D(u(y,t) - u(x,t))$ where $D: \R \ra \R$ is a Lipschitz-continuous function.  \cite[Theorem 4.1]{medvedev2014nonlinear} holds for nonlinear diffusion.
\end{remark}

\begin{remark}\label{remark: medvedev for any convex function}
One can generalize Theorem \ref{thm: diffusion convergence} to any convex function of $u_n-u$ (in Theorem \ref{thm: diffusion convergence}, the convex function is $\|\cdot \|_p$). A convex function is needed because, in order to bound term $(\mathrm{I})$, the function $\phi$ must be monotonic. 

\end{remark}

When considering $L^\infty$ graphons, we have a slightly different result under stronger assumptions.

\begin{theorem}\label{thm: L infty diffusion convergence}
    Let $p\in [1,\infty)$ and $T>0$. Assume that $W_n$ and $W$ satisfy \Cref{assump:W2}i). Let $u_n \in C^1([0,T];L^p((0,1)))$ be the solution to the graph diffusion equation \eqref{eq: graph diffusion equation} and let $u \in C^1([0,T];L^\infty((0,1)))$ be the solution to the graphon diffusion equation \eqref{eq: graphon diffusion equation}; let their initial conditions satisfy \Cref{assump: u0}iii). 
    Then, 
    \begin{equation*}
        \|u(\cdot, t) - u_n(\cdot, t)\|_\infty \leq \|u(\cdot, 0) - u_n(\cdot, 0)\|_\infty + 2 t \|u(\cdot,0)\|_\infty \|W_n-W\|_\infty\,.
    \end{equation*}
\end{theorem}

\begin{proof}
     The proof begins and proceeds identically to that of \Cref{thm: diffusion convergence} until (and including) equation \eqref{eq: bound II}. The remaining steps differ from the proof of \Cref{thm: diffusion convergence} in that we bound terms using the quantity $\|W_n-W\|_p$ instead of $\|W_n-W\|_\square$. We then take $p \to \infty$. 
    
    Beginning with the proof of \Cref{thm: diffusion convergence} up until equation \eqref{eq: bound II}, we obtain 
\begin{align*}
    \frac{1}{p} \frac{\partial}{\partial t} \|e_n(\cdot, t)\|_p^p 
    &\leq (\mathrm{II}) \,, 
\end{align*}
where 
\begin{equation*}
    (\mathrm{II}) = \int_0^1 \varphi(e_n(x,t))
    \, \cL^{W_n-W}u(x,t) \, dx
\end{equation*}
is the same term (II) that was defined in \eqref{eq: (I) + (II)}. We now bound (II) in terms of $\|W_n-W\|_p$, as opposed to $\|W_n-W\|_\square$ as was done in \Cref{thm: diffusion convergence}. 
\begin{align*}
    (\mathrm{II}) &= - \frac{1}{2} \int_0^1 \int_0^1 (W_n-W)(x,y) (u(y,t)-u(x,t)) \big( \varphi(e_n(y,t))-\varphi(e_n(x,t)) \big) \, dydx
    \\
    &\leq \| u(\cdot,t) \|_\infty \left( \int_0^1 \int_0^1 |(W_n-W)(x,y)|^p dydx \right)^{1/p} 
    \\
    &\qquad \qquad \cdot \left( \int_0^1 \int_0^1 |\varphi(e_n(x,t)) - \varphi(e_n(y,t)) |^{q} dy dx  \right)^{1/q}\\
    &\leq \| u(\cdot,t) \|_\infty \| W_n - W\|_p \big( 2 \| e_n(\cdot,t) \|_p^{p-1} \big)\\
    &\leq 2\| u(\cdot,0) \|_\infty \| W_n - W\|_p  \| e_n(\cdot,t) \|_p^{p-1} \numberthis \label{eq: term ii lastline infty}
    \end{align*}
    where the last line uses \Cref{lem: stability of semigroup}. 
    Combining the bound \eqref{eq: term ii lastline infty} with the identity \eqref{eq: chain rule derivative of Lp}, we obtain
    \begin{equation}\label{eq: l^infty diffusion error bound}
        \frac{\partial}{\partial t} \| e_n(\cdot,t) \|_p \leq 2 \| u(\cdot,0) \|_\infty \| W_n - W\|_p \,. 
    \end{equation}
    The desired result follows from applying Gronwall's inequality to \eqref{eq: l^infty diffusion error bound} and then taking $p \to \infty$.
\end{proof}

%%%%%%%%%%%%%%%%%%%%%%%%%%%%%%%
\subsection{Graphon reaction--diffusion equation}\label{sec: rd equation convergence}

Consider the RD equation \eqref{eq: classical RD}. 
A classical RD equation uses $L = \Delta$, the Laplacian operator \cite{britton1986reaction}. 
We define the graph and graphon RD equations by defining $L$ to be the graph and graphon diffusion operators, respectively. The graph RD equation, corresponding to a graph $W_n$, is 
\begin{equation} \label{eq: graph RD} 
    \frac{\partial u_n}{\partial t} (x,t) = \calL^{W_n} u_n(x,t) + \Phi(u_n(x,t))  \,,
\end{equation}
and the graphon RD equation, corresponding to a graphon $W$, is 
\begin{equation} \label{eq: graphon RD}
    \frac{\partial u}{\partial t} (x,t) = \calL^W u(x,t) + \Phi(u(x,t)) \,. 
\end{equation}

We consider the initial value problems associated with these equations and we obtain convergence of the solution of \eqref{eq: graph RD} to the solution of \eqref{eq: graphon RD} as $n\ra \infty$. 
We now state a well-posedness result for the equations, which follows from the boundedness of the graphon diffusion operator \cite[Theorems 1.4, 1.5 of Chapter 6.1]{pazy2012semigroups}. 
%Theorem 1.4 says: if $\Phi$ is locally Lipschitz then the IVP has a unique mild solution on $[0,t_{max})$ for some $t_{max}$. 
%Theorem 1.5 says: if in addition, $\Phi$ is continuously-differentiable in $t$, then the mild solution is a classical solution. (for us, $\Phi$ is constant in t)
\begin{theorem}\label{thm:RdWellPosed-infty}
    Assume that $\Phi$ satisfies \Cref{assump: phi}i).
    Let $u_0 \in L^\infty((0,1))$ with $\| u_0 \|_\infty \leq M$ for some constant $M > 0$. Then there exists $T_{max} > 0$ depending on $M$ and on the Lipschitz constant of $\Phi$ on $[-M,M]$ such that the graphon RD equation \eqref{eq: graphon RD} with initial condition $u(x,0) = u_0(x)$ has a unique classical solution $u \in C^1([0,T_{max}];L^\infty((0,1)))$ that satisfies the mild formulation 
    \begin{equation}\label{eq:RDSemigroupFormula}
        u(x,t) = e^{t \cL^W } u_0 + \int_0^t e^{ (t-s)\cL^W} \Phi(u(x,s)) \, ds\, .
    \end{equation}
    Moreover, if $T_{max} < \infty$, then $\lim\limits_{t \to T_{max}} \| u(\cdot,t) \|_\infty = + \infty$.
\end{theorem}

With a stronger assumption on $\Phi$, we obtain the following global existence result for $p \in [1,\infty)$ \cite[Theorems 1.5, 1.7 of Chapter 6.1]{pazy2012semigroups}. 
%Theorem 1.5 stated above
%Theorem 1.7 says that if $\Phi$ is uniformly Lipschitz then the IVP has a unique classical solution on $[0,T]$. 

\begin{theorem}\label{thm:RdWellPosed-p}
    Suppose $\Phi$ satisfies \Cref{assump: phi}iii), let $p \in [1,\infty)$, and let $u_0 \in L^p((0,1))$. Then for any $T > 0$, the equation \eqref{eq: graphon RD} with initial condition $u(x,0) = u_0(x)$ has a unique classical solution $u \in C^1([0,T];L^p((0,1)))$ that satisfies the mild formulation \eqref{eq:RDSemigroupFormula}.
\end{theorem}

When $p = \infty$, the local well-posedness can be extended to global well-posedness, even when $\Phi$ is only locally Lipschitz continuous, by proving a contraction estimate on $\| u(\cdot,t) \|_\infty$ which prevents the blow-up of the solution at any finite time. \Cref{thm:MaxPrinciple} states this fact; its proof is based on arguments from \cite[Theorem 2.3]{du2021maximum}.

\begin{theorem}\label{thm:MaxPrinciple}
    Let $\Phi$ satisfy \Cref{assump: phi}i) and \Cref{assump: phi}ii).
    Let $u$ be the unique classical solution to the graphon RD equation \eqref{eq: graphon RD} on $[0,T_{max}]$ with initial condition $u_0 \in L^\infty((0,1))$. If 
    \begin{equation}
    \label{eq:initial bound}
        M_1 \leq u_0(x)=u(x,0) \leq M_2\,, \qquad \text{ for a.e. } x \in (0,1)\,, 
    \end{equation}
    then for all $t\in [0,T_{max}]$, 
    \begin{equation*}
        M_1 \leq u(x, t) \leq M_2\,, \qquad \text{ for a.e. } x \in (0,1). 
    \end{equation*}
    Consequently, $u$ can be extended to a unique classical solution of \eqref{eq: graphon RD} on $[0,T]$ for any arbitrary $T > 0$.
\end{theorem}
Note that Theorem \ref{thm:MaxPrinciple} holds for both \eqref{eq: graph RD} and \eqref{eq: graphon RD} because \eqref{eq: graph RD} is a special case of \eqref{eq: graphon RD}.

\begin{theorem}[Convergence of solutions of graph RD to graphon RD equations]\label{thm: graph rd convergence to graphon rd}
    Let $p \in [1,\infty)$ and $T>0$.
    Let $\Phi$ satisfy \Cref{assump: phi}iii) and let $W_n$ and $W$ satisfy \Cref{assump:W2}ii). For each $W_n$, let  $u_n \in C^1([0,T];L^p((0,1)))$ be the unique classical solution to the graph RD equation \eqref{eq: graph RD} and let $u \in C^1([0,T];L^\infty((0,1)))$ be the unique classical solution to the graphon RD equation \eqref{eq: graphon RD}; let their initial conditions satisfy \Cref{assump: u0}iii). 
    Then, 
\begin{equation*} 
    \|u_n(\cdot, t) - u(\cdot, t)\|_p \leq  e^{Kt} \|u_n(\cdot, 0)-u(\cdot, 0)\|_p  + \frac{4M \left( e^{K t} - 1\right) }{K} \|W_n-W\|_\square^{1/p} \,,
\end{equation*}
where $M = \| u(\cdot,0) \|_\infty$.
\end{theorem}

\begin{proof}
Let $e_n(x,t) = u_n(x,t) - u(x,t)$. 
\begin{align*}
    \frac{\partial e_n}{\partial t}(x,t) &= \frac{\partial u_n}{\partial t}(x,t) - \frac{\partial u}{\partial t}(x,t)
    \\
    &= \int_0^1 W_n(x,y)\big(u_n(y,t) - u_n(x,t)\big) \, dy - \int_0^1 W(x,y)\big(u(y,t) - u(x,t)\big) \, dy 
    \\
    &\quad \quad \quad \quad + \Phi(u_n(x,t)) - \Phi(u(x,t)) 
    \\
    &= \int_0^1 W_n(x,y)\big(e_n(y,t)-e_n(x,t)\big)\, dy + \int_0^1 (W_n-W)(x,y) \big(u(y,t)-u(x,t)\big)\, dy 
    \\
    &\quad \quad \quad \quad + \Phi(u_n(x,t)) - \Phi(u(x,t)) \,.\numberthis \label{eq: partial deriv ac}
\end{align*}
Then, $e_n$ satisfies the equation
\[ \frac{\partial e_n}{\partial t} (x,t) = \cL^{W_n} e_n + \cL^{W_n-W} u + \Phi(u_n) - \Phi(u).\]

Using the identity \eqref{eq: time deriv of lp norm} and the notation $\varphi(a) = |a|^{p-2}a$, we obtain
\begin{align*}
    \frac{1}{p} \frac{\partial}{\partial t} \|e_n(\cdot, t)\|_p^p 
    &= \frac{1}{p} p \left \langle \varphi(e_n(\cdot,t)), \frac{\partial}{\partial t} e_n(\cdot, t) \right \rangle 
    \\
    &= \int_0^1 \varphi(e_n(x,t)) \cL^{W_n} e_n \, dx  + \int_0^1 \varphi(e_n(x,t)) \cL^{W_n-W} u \, dx 
    \\
    & \qquad + \int_0^1 \varphi(e_n(x,t)) \left(\Phi(u_n) - \Phi(u)\right) dx
    \\
    &=: (\mathrm{I}) + (\mathrm{II}) + (\mathrm{IV})\,,
\end{align*}
where the quantity (IV) is so named to avoid confusion with the term (III) from the proof of \Cref{thm: diffusion convergence} (in equation \eqref{eq: term iii lastline}). The terms (I) and (II) are the same as those from \Cref{thm: diffusion convergence}. 
Using the identity \eqref{eq: chain rule derivative of Lp}, we get  
\begin{equation} \label{eq: rd error derivative three terms}
    \frac{\partial}{\partial t} \|e_n(\cdot, t)\|_p \leq  \|e_n(\cdot, t)\|^{1-p}_p \left((\mathrm{I}) + (\mathrm{II}) + (\mathrm{IV}) \right) \,.  
\end{equation}

We now bound the right-hand side of \eqref{eq: rd error derivative three terms}. The calculations in the proof of \Cref{thm: diffusion convergence} (see equations \eqref{eq: term ii lastline} and \eqref{eq: term iii lastline})
give 
\begin{equation}\label{eq: rd convergence term i and ii}
    \|e_n\|^{1-p}_p \left((\mathrm{I}) + (\mathrm{II}) \right) \leq 4 \| u(\cdot,0) \|_\infty \|W_n-W\|_\square^{1/p} = 4 M \|W_n-W\|_\square^{1/p} \,. 
\end{equation}
To bound $(\mathrm{IV})$, let $q = p/(p-1)$. H{\"o}lder's inequality and the fact that $\Phi$ is $K$-Lipschitz give 
\begin{align*}
    \|e_n\|^{1-p}_p (\mathrm{IV}) &= \|e_n\|^{1-p}_p \int_0^1 \varphi(e_n(x,t)) \big( \Phi(u_n(x,t)) - \Phi(u(x,t)) \big) \, dx
    \\
    &\leq \|e_n\|^{1-p}_p \, \|e_n^{p-1}\|_q \, \|\Phi(u_n(x,t)) - \Phi(u(x,t))\|_p 
    \\
    &= \|e_n\|^{1-p}_p \, \|e_n\|_p^{p-1} \, \|\Phi(u_n(x,t)) - \Phi(u(x,t))\|_p 
    \\
    &\leq K \|e_n(\cdot, t)\|_p \,. \numberthis \label{eq: rd convergence term iii}
\end{align*}

Combining the bounds \eqref{eq: rd convergence term i and ii} and \eqref{eq: rd convergence term iii}, we have that 
\begin{equation} \label{eq: RD to apply gronwall to}
    \frac{\partial}{\partial t} \|e_n(\cdot, t)\|_p \leq 4 M \|W_n-W\|_\square^{1/p} + K \|e_n(\cdot, t)\|_p \,. 
\end{equation}
Gronwall's inequality applied to \eqref{eq: RD to apply gronwall to} gives 
\begin{equation*} 
    \|e_n(\cdot, t)\|_p \leq  \|e_n(\cdot, 0)\|_p \, e^{Kt} + 4 M \|W_n-W\|_\square^{1/p} \frac{\left( e^{Kt} -1\right) }{K}\,. 
\end{equation*}
\end{proof}

\begin{remark}
\label{rem:factor}
    Theorem \ref{thm: graph rd convergence to graphon rd} reduces to Theorem \ref{thm: diffusion convergence} in the case $K=0$ because $\frac{\left( e^{K t} -1\right)}{K} \ra t$ as $K \ra 0$. Note that if $\Phi$ is a monotone function, in which case the solution of the graph and graphon RD equations are each gradient flows of a convex energy, then the solutions have a long-time asymptotic stability. In such a case, one may eliminate the exponential factor $e^{2Kt}$ or even prove the conclusion with a negative exponent $K$.  
\end{remark}

\begin{remark} Just like in the diffusion equation case (see Remark \ref{remark: medvedev for any convex function}), one can generalize Theorem \ref{thm: graph rd convergence to graphon rd} to show that any convex function of $u_n-u$ converges to zero as $n\ra \infty$. 
\end{remark}

Now we state and prove the result in the case that $W_n$ and $W$ are $L^\infty$ graphons.

\begin{theorem}\label{thm: graph rd convergence to graphon rd Linfty}
   Let $T>0$.
    Let $\Phi$ satisfy \Cref{assump: phi}i).
    Let $W_n$ and $W$ satisfy \Cref{assump:W2}i). For each $W_n$, let $u_n \in C^1([0,T];L^\infty((0,1)))$ be the unique classical solution to the graph RD equation \eqref{eq: graph RD} and let $u \in C^1([0,T];L^\infty((0,1)))$ be the unique classical solution to the graphon RD equation \eqref{eq: graphon RD}; let their initial conditions satisfy \Cref{assump: u0}iv). If $u_n(\cdot,0)$ and $u(\cdot,0)$ both satisfy \eqref{eq:initial bound}, then
\begin{equation*} 
    \|u_n(\cdot, t) - u(\cdot, t)\|_\infty \leq  e^{Kt} \|u_n(\cdot, 0)-u(\cdot, 0)\|_\infty  + \frac{2M \left( e^{K t} - 1\right) }{K} \|W_n-W\|_\infty\,,
\end{equation*}
where $M = \| u(\cdot,0) \|_\infty$ and $K$ is the Lipschitz constant of $\Phi$ on the interval $[M_1,M_2]$.
\end{theorem}

\begin{proof} 
    Let $e_n(x,t) = u_n(x,t) - u(x,t)$. We obtain equation \eqref{eq: rd error derivative three terms} in the same way as it was obtained in the proof of \Cref{thm: graph rd convergence to graphon rd}. Namely, we obtain
    \begin{align*}
    \frac{\partial}{\partial t} \|e_n(\cdot, t)\|_p 
    &= \|e_n\|^{1-p}_p  \left( (\mathrm{I}) + (\mathrm{II}) + (\mathrm{IV}) \right).
\end{align*}
Using the bounds \eqref{eq: bound on (I)} and \eqref{eq: term ii lastline infty}, we have that
\begin{equation*}
    \|e_n(\cdot, t)\|^{1-p}_p \left( (\mathrm{I}) + (\mathrm{II}) \right) \leq  2\|u(\cdot,0)\|_\infty \|W_n-W\|_p \,. 
\end{equation*}
It remains to bound $\|e_n\|^{1-p}_p (\mathrm{IV})$. By \Cref{thm:MaxPrinciple}, we know that $u_n(x,t)$, $u(x,t) \in [M_1,M_2]$ for almost every $x$ and all $t \in [0,T]$. Thanks to \Cref{assump: phi}i), we obtain 
    \begin{equation*}
       \big| \Phi(u_n(x,t)) - \Phi(u(x,t)) \big| \leq K |e_n(x,t)|, \qquad \text{a.e. in } (0,1), t \in [0,T].
    \end{equation*}
    As a result, we can obtain the same bound \eqref{eq: rd convergence term iii} for $\|e_n\|^{1-p}_p (\mathrm{IV})$, even though $\Phi$ is now assumed to be only locally Lipschitz. Finally, we obtain 
    \begin{equation*}
        \frac{\partial}{\partial t} \| e_n(\cdot,t) \|_p \leq 2 \|u(\cdot,0)\|_\infty \| W_n - W \|_p + K \| e_n(\cdot,t) \|_p \,.
    \end{equation*}
    The desired result then follows from Gronwall's inequality and taking $p \to \infty$.
\end{proof}

\section{Interacting particle process} \label{sec: IPS}

In this section, we describe an interacting particle process  on graphs that consists of a continuous-time random walk superimposed with a birth-death process. For a sequence of graphons $W_n$ converging in cut norm to a graphon $W$, we show that the sequence of particle processes on $W_n$ converges in probability to the solution of a nonlocal reaction--diffusion (RD) equation, where the nonlocal diffusion is weighted by the graphon $W$.

Consider a graph $W_n$ with $n$ nodes. Suppose that at time $t=0$, there is some number of particles on each node. Let $m_k(t)$ be the number of particles on node $k$ at time $t$, and let $\vec{m}(t) = (m(t), m_1(t), \dots, m_n(t))$. Let $\ell > 0$ be a parameter proportional to the average number of particles per node at $t=0$, so that $m_k/\ell$ is a density of particles. In the large-graph limit, $\ell$ will act as a scaling parameter by growing to $\infty$ as $n$ grows to $\infty$.

Let $e_i \in \R^n$ be the $i$th unit vector. As is described in \cite{watanabe2022continuum}, we let $\vec{m}$ undergo the stochastic transitions
\begin{equation}\label{eq: def m jumps}
\left\{\begin{array}{l}
    \vec{m} \rightarrow \vec{m}+e_i-e_k  \quad \text { at rate } m_k W_n\left(k/n, i/n\right)n^{-1}, \text{ for all } i,k \in [n] \,, \\
    \vec{m} \rightarrow \vec{m}+e_k \quad \quad \quad \text { at rate } \ell b\left(m_k/\ell\right), \\
    \vec{m} \rightarrow \vec{m}-e_k \quad \quad \quad \text { at rate } \ell d \left(m_k/\ell\right),
    \end{array}\right.
\end{equation} 
where $b, d: \R \ra \R$ are uniformly Lipschitz-continuous birth and death rates, respectively.
Without loss of generality, we assume that $b(0) = d(0) = 0$. 
Moreover, we use the notation that $\Phi(x):= b(x) - d(x)$.  The rate $m_k W_n(k/n,i/n)$ of jumps from site $k$ to site $i$ is proportional to the number of particles at site $k$ and to the density of edges between sites $k$ and $i$. 
We also assume that $\vec{m}(t)$ is a c{\`a}dl{\`a}g process on some probability space.
We now define the stochastic process $X^n: (0,1) \times [0,\infty) \ra \R$ as the scaled process 
\begin{equation}\label{eq: def X^n}
    X^n(x,t) = \frac{m_k(t)}{\ell} \; \text { for } \; x \in I_k \,,
\end{equation}
for each node $k\in [n]$. The quantity $X^n(x,t)$ is the density of particles at node $k$. Note that we suppress the dependence of $X^n$ on $\ell$. We also sometimes suppress the dependence of $X^n$ on $x$, and write it as $X^n(t)$. Because $X^n(t)$ is piecewise-constant on $(0,1)$, we are able to apply the graph Laplacian semigroup operator to $X^n(t)$.

Define $\mathcal{F}_{t}^{n,\ell}$ to be the $\sigma$-algebra generated by $\vec{m}(t)$. 
Suppose that $\tau$ is a stopping time that satisfies
\begin{equation}\label{eq:GenericStopTime}
    \sup_{t\in [0,T]} \sup_{k \in \{1,\ldots,n \} } \pmb{1}_{\{ \tau > 0 \} } m_k( t \wedge \tau) = C(T,n,\ell) < \infty, \qquad \forall \, T, n,\ell > 0,
\end{equation}
where we denote $t \wedge \tau = \min \{ t, \tau \}$. The condition \eqref{eq:GenericStopTime} guarantees that the particle process is finite at the stopping time. 

\Cref{lem: martingale}, which we state below, identifies the infinitesimal generator of the process $\vec{m}$. The generator resembles the graphon RD operator $\calL^{W_n} + \Phi$. Using Dynkin's formula, we then provide a mean-zero martingale associated to $\vec{m}$.

\begin{lemma}\label{lem: martingale}
    Let $A : \R^n \to \R^n$ be the function defined by
    \begin{equation}\label{eq: generator A}
        A(\vec{x}) = \sum_{k=1}^{n} \left[ \frac{x_k}{n} \sum_{i = 1}^n W(\frac{k}{n},\frac{i}{n}) (e_i - e_k) + \ell e_k \big( b(\frac{x_k}{\ell}) - d(\frac{x_k}{\ell}) \big) \right]\,, 
    \end{equation}
    where $\vec{x} = (x_1,\ldots,x_n) \in \R^n$. Then, the process 
    \begin{equation}\label{eq: M}
        \vec{M}(t) = \vec{m}(t \wedge \tau) - \vec{m}(0) - \int_0^{t \wedge \tau} A(\vec{m}(s)) ds,
    \end{equation}
    is a mean-zero martingale. 
\end{lemma}

\begin{proof}
    $A$ is the infinitesimal generator of the process $\vec{m}$. 
    See the Appendix \ref{sec: appendix} for a proof of this fact. 
    
    To prove the lemma, we must show that 
    \begin{itemize}
        \item[i)] $\mathbb{E}[|\vec{M}(t)|] < \infty$ for all $t \in[0,T]$, and
        \item[ii)] $\mathbb{E}[ (\vec{M}(t) - \vec{M}(s)) 1_F ] = 0$ for all $s < t$ and for all sets $F \in \mathcal{F}_s^{n,\ell}$. 
    \end{itemize}
    Statement i) holds because $\vec{m}(t\wedge \tau)$ is bounded due to the stopping time property \eqref{eq:GenericStopTime} and because $A: \R^n \ra \R^n$ is a bounded operator. The operator $A$ is bounded because the graph diffusion operator is bounded (see equation \eqref{eq:DiffusionOperatorBd2}) and because the birth-death function $b-d$ is Lipschitz on its domain $[0,1]$.

    We now show statement ii). Denote $\ex^{\vec{x}}[\, \cdot\, ] = \ex[\, \cdot \, | \vec{m}(0) = \vec{x}]$. Due to the definition \eqref{eq: M}, we have 
    \begin{align*} 
        \ex[\vec{M}(t\wedge \tau)] - \ex[\vec{M}(s\wedge \tau) ]
        &= \ex^{\vec{x}}[\vec{m}(t\wedge \tau)] - \ex^{\vec{x}}[\vec{m}(s\wedge \tau)] - \ex^{\vec{x}}\left[\int_{s\wedge \tau}^{t\wedge \tau} A\vec{m}(z) dz \right]\,. 
    \end{align*}
    By Dynkin's formula, the right-hand-side is equal to zero. 
    Next, we show that $\vec{M}(t)$ is a mean-zero martingale. By Dynkin's formula and the fact that $A$ is the generator of $\vec{m}$, we have 
    \begin{equation*} 
        \ex^{\vec{x}}[f(\vec{m}(t\wedge \tau)] = f(\vec{x}) + \ex^{\vec{x}}\left[\int_0^{t \wedge \tau} Af(\vec{m}(s))ds \right] \,. 
    \end{equation*}
    Taking $f = Id$ and $\vec{x} = \vec{m}(0)$, we have $\ex[\vec{M}(t)] = 0$. 
\end{proof}

This result implies that the quantity $Z^n(t \wedge \tau)$, where
\begin{equation}\label{eq: def Z^n}
    Z^n(t) = X^n(t) - X^n(0) - \int_0^t \calL^{W_n} (X^n(s)) \, ds - \int_0^t \Phi(X^n(s))\, ds, \qquad t \geq 0,
\end{equation}
is an $\mathcal{F}_t^{n,\ell}$ mean-zero martingale. (See Appendix \ref{appendix: Z^n is a martingale} for a proof of this fact.)

Using \eqref{eq: def Z^n}, we can write $X^n$ as 
\begin{equation}\label{eq: alt def X^n}
    X^n(t) = X^n(0) + \int_0^t \calL^{W_n} (X^n(s)) \, ds + \int_0^t \Phi(X^n(s))ds + Z^n(t) \,. 
\end{equation}
We apply the method of variation of parameters \cite[Chapter 3]{engel2000one}
to \eqref{eq: alt def X^n}---that is, we apply the operator $e^{-\calL^{W_n} t} = [e^{t\calL^{W_n} }]^{-1}$ to both sides, integrate both sides, and rearrange---to get 
\begin{equation}\label{eq: alt alt def X^n}
    X^n(t) = e^{t \calL^{W_n}} X^n(0) + \int_0^t e^{ (t-s)\calL^{W_n}}\Phi(X^n(s)) \, ds + Y^n(t)\,,
\end{equation}
where we have defined 
\begin{equation}\label{eq: def Y^n}
    Y^n(t) = \int_0^t e^{ (t-s) \calL^{W_n}} \, dZ^n(s).
\end{equation}
The stochastic integral $Y^n(t)$ is a Stieltjes integral, and it is well-defined because $Z^n(s,kn^{-1})$ is of bounded variation in $s$ and because $t \mapsto e^{t\calL^{W_n}}$ can be viewed as a continuous $n\times n $ matrix-valued function.

Observe that \eqref{eq: alt alt def X^n} resembles the mild formulation \eqref{eq:RDSemigroupFormula} of the graphon RD equation, but \eqref{eq: alt alt def X^n} has an additional stochastic term. In the following subsection, we state a law of large numbers theorem. 

\subsection{Law of large numbers}
The main result of this section is Theorem \ref{thm: lln in l2}, which shows that the quantity $\|X^n(t) - u(t)\|_2$ converges to zero in probability (where $u$ is the solution of the graphon RD equation \eqref{eq: graphon RD}).

To prove \Cref{thm: lln in l2}, we show that each of the quantities $\|X^n(t) - u_n(t)\|_2$ and $\|u_n(t) - u(t)\|_2$ converge to zero in probability. The stability and convergence results shown in Sections \ref{sec: graphon diffusion} and \ref{sec: rd equation convergence} allow us to prove that the term $\|u_n(t) - u(t)\|_2$ converges to zero as $W_n \sqra W$. We then rewrite the quantity $\|X^n(t) - u_n(t)\|_2$ using the semigroup formulations \eqref{eq: alt alt def X^n} and \eqref{eq:RDSemigroupFormula} of the stochastic process and of the graphon RD equation, respectively. To show that the stochastic term $Y^n$ converges to zero in probability, we use \Cref{lma:Martingale2}, which is an application of Dynkin's formula to our stochastic process \eqref{eq: def m jumps}.

\begin{theorem}\label{thm: lln in l2}
Let $W_n$ be a sequence of graphs that converge in cut norm to an $L^1$ graphon $W$. I.e., assume that both $W_n$ and $W$ satisfy \Cref{assump:W2}ii). Assume that the reaction term $\Phi = b - d$ satisfies \Cref{assump: phi}iv). Let $u \in C^1([0,T];L^\infty((0,1)))$ be the solution to the graphon RD equation \eqref{eq: graphon RD} with initial condition $u_0 \in L^\infty((0,1))$. Suppose that 
\begin{enumerate}[(I)]
    \item $\|X^n(0) - u_0 \|_2 \ra 0$ in probability as $n\ra\infty$, and 
    \item the dependence of $\ell$ on $n$ is such that $\ell \to \infty$ as $n \to \infty$.
\end{enumerate}
Then, for $T > 0$, 
\begin{equation}
    \sup_{0 \leq t \leq T} \|X^n(t) - u(t)\|_2 \ra 0 \quad \text{ in probability as } n \ra \infty \,. 
\end{equation}
\end{theorem}

Let us introduce the notation 
\begin{equation}
    \Lambda X^n(t) = \|X^n(t) - X^n(t^-) \|_\infty \,,
\end{equation}
which gives the gives the jump size of $X^n$ at time $t$. We also use the notation $\Lambda$ to describe the jump sizes of other stochastic jump processes. Note that the process $X^n$ always has a jump size of $\frac{1}{\ell}$ and the process $m_k$ always has a jump size of $1$.

The following lemma is similar to \cite[Lemma 2.2]{blount1987comparison} and \cite[Lemma 4.1]{watanabe2022continuum}). Its statement is similar to that of \Cref{lem: martingale}, but it replaces quantities of the form $m(t\wedge \tau) - m(0)$ with a discrete sum over the jumps of $m$ between time $0$ and $t\wedge \tau$. 
\begin{lemma}\label{lma:Martingale2}
    Let $\vec{m}$ undergo the transitions described in \eqref{eq: def m jumps}. For each $k \in [n]$, the quantity 
    \begin{align*}
        &\sum_{0\leq s \leq t \wedge \tau} (\Lambda m_k(s))^2 
        - \int_0^{t \wedge \tau} \frac{1}{n} \sum_{\substack{j = 1 \\ j \neq k}}^{n} W_n(k/n,j/n) (m_j(s) + m_k(s) ) ds\\
        &\qquad  -
        \int_0^{t \wedge \tau} \ell b(m_k(s)/\ell) + \ell d(m_k(s)/\ell) ds
    \end{align*} 
    is a mean-zero martingale. 
\end{lemma}

\begin{proof}
    We begin by computing the infinitesimal generator for the process 
    \begin{equation*}
        N(t) := \sum_{0 < s \leq t \wedge \tau} (\Lambda m_k(s))^2 \,. 
    \end{equation*}
    Because of the definition \eqref{eq: def m jumps} of the process $\vec{m}$, we have that the $k$th node can undergo the following transitions 
    \begin{align*}
        &m_k \ra m_k+1 \text{ with rate } \ell b (m_k/\ell) \\
        &m_k \ra m_k-1 \text{ with rate } \ell d (m_k/\ell) \\
        &m_k \ra m_k+1 \text{ with rate } \sum_{j\neq k} \frac{m_j}{n} W_n(k/n,j/n) \\
        &m_k \ra m_k - 1 \text{ with rate } \sum_{j\neq k} \frac{m_k}{n} W_n(k/n,j/n) \,. 
    \end{align*}
    Therefore, when $N(t)$ changes, it is always the change $N(t) \ra N(t) + 1$. The rate at which the change $N(t) \ra N(t) + 1$ occurs is the sum of the rates of all the transitions for $m_k$ given above, and therefore the infinitesimal generator $\tilde{A}$ of the process $N(t)$ is \cite{norris1998markov} 
    \begin{equation*}
        \tilde{A} (N(t)) = \ell b(m_k(t)/\ell) + \ell d (m_k(t))/\ell) + \sum_{j\neq k} \left(\frac{m_j(t)}{n} + \frac{m_k(t)}{n} \right) W_n(k/n,j/n) \,. 
    \end{equation*}
   
    Using Dynkin's formula, we obtain that 
    \begin{equation*} 
        N(t\wedge \tau) - N(0) - \int_0^{t\wedge \tau} \tilde{A} (N(s))ds 
    \end{equation*}
    is a mean-zero martingale. The statement of \Cref{lma:Martingale2} follows. 
\end{proof}

%We use four lemmas: Lemma \ref{lem: stability of semigroup} states that $e^{\calL^{W_n} t}$ is nonexpansive, Lemma \ref{lem: semigroup convergence} states that $e^{\calL^{W_n}}$ converges to $e^{\calL^W}$, and Lemma \ref{lem: mbp} ensures that the solution of the graphon RD equation is bounded in $L^\infty$. 

\begin{proof}[Proof of Theorem \ref{thm: lln in l2}]
The goal is to show that 
\begin{equation}\label{eq: WTS} 
    \lim_{n\ra\infty} P \Big( \sup_{t\in[0,T]}\|X^n(t)-u(t)\|_2 > \ep \Big) = 0\,. 
\end{equation}
Let $u_n$ be the solution to the graph reaction--diffusion equation \eqref{eq: graph RD} with the initial condition $u_n(0)$. 
Applying the triangle inequality, we get 
\begin{equation} \label{eq: TI for X^n - u}
    \|X^n(t) - u(t)\|_2 \leq \|X^n(t) - u_n(t)\|_2 + \|u_n(t) - u(t)\|_2\,. 
\end{equation}
Therefore, 
\begin{align*}
    P\left(\sup_{t\in[0,T]}\|X^n(t)-u(t)\|_2 >\ep\right) &\leq P\left(\sup_{t\in[0,T]}\|X^n(t)-u_n(t)\|_2 >\ep\right) 
    \\
    &\quad \quad + P\left(\sup_{t\in[0,T]}\|u_n(t)-u(t)\|_2 >\ep\right) 
    \\
    &= (A) + (B)  \,. 
\end{align*}
We proceed by showing that both $(A)$ and $(B)$ converge to zero as $n \ra \infty$. 
Term $(B)$ converges to zero as $n\ra \infty$ by the following argument. The estimate of \Cref{thm: graph rd convergence to graphon rd} implies that $\sup_{t \in [0,T] } \|u_n(t)-u(t)\|_2 \to 0$ as $n \to \infty$, because $u_n(x,0) = u(x,0) = u_0(x)$ and because $\|W_n-W\|_\square \ra 0$ as $n\ra \infty$ by assumption. It follows that $(B) \ra 0$ as $n\ra \infty$.

Next, we bound $(A)$. For a fixed $\ep > 0$ define the stopping time 
\begin{equation}\label{def:tau}
        \tau = \inf_{t\in [0,T]} \left\{t : \left\|X^n(t)-u_n(t)\right\|_2 > \ep \right\}.
\end{equation} 
Thus, we can rewrite 
\begin{equation}\label{eq: equivalent (A)}
    \begin{split}
        (A) = P\left(\sup_{t\in[0,T]}\|X^n(t)-u_n(t)\|_2 > \ep\right) 
        &\leq P\left(\sup_{t\in[0,T]}\|X^n(t \wedge \tau)-u_n(t \wedge \tau)\|_2 \geq \ep\right) 
    \end{split}
\end{equation}
because of the following. 
If $\sup_{t\in[0,T]} \| X^n(t) - u_n(t) \|_2 \geq \ep$, then $\tau \leq T$, and therefore
$\sup_{t \in [0,\tau]} \|X^n(t) - u_n(t) \|_2 \geq \ep$ by definition of $\tau$, which implies $\sup_{t \in [0,T]} \| X^n(t \wedge \tau) - u_n(t \wedge \tau) \|_2 \geq \ep$.

Therefore, it suffices to 
show that the right-hand side of \eqref{eq: equivalent (A)} converges to $0$ as $n \to \infty$. 
We claim that 
\begin{equation}\label{eq:stoppingtimestochasticbound}
    \sup_{t \in [0,T]} \| X^n(t \wedge \tau) \|_2 \leq C\,, 
\end{equation}
where $C$ is a constant that depends only on $\| u_0 \|_2$ and $T$. To prove this claim, we first obtain an $L^2$ bound on $u_n$. Using the semigroup formulation \eqref{eq:RDSemigroupFormula} of $u$
and the contractive bound on the semigroup $e^{t\cL^{W_n}}$ from \eqref{eq:Semigroup:LpStab}, together with the assumption that $\Phi$ is uniformly Lipschitz imply that
    \begin{equation*}
        \| u_n(\cdot,t) \|_2 \leq \| u_0 \|_2 + K \int_0^t \| u_n(\cdot,s) \|_2 \, ds.
    \end{equation*}
    Then Gronwall's inequality implies that
    \begin{equation}\label{eq: l2 bound on u_n}
        \| u_n(\cdot,t) \|_2 \leq \| u_0 \|_2 e^{ Kt } 
        \leq \|u_0\|_2 \, e^{KT} \qquad \forall t \in [0,T] \,. 
    \end{equation}
    If $t < \tau$, then
    \begin{align*}
        \| X^n(t \wedge \tau) \|_2 &\leq \| u_n(t \wedge \tau) \|_2 + \|X^n(t \wedge \tau) - u_n(t \wedge \tau) \|_2 
        \\
        &\leq \| u_0 \|_2 \, e^{KT} + \epsilon 
    \end{align*}
    due to \eqref{eq: l2 bound on u_n} and the definition of the stopping time. 
    If $t = \tau$, then
    \begin{align*}
        \| X^n(\tau) \|_2 \, \leq\,  \| X^n(\tau -) \|_2 + \| \Lambda X^n(\tau) \|_2 \, \leq \, \| u_0 \|_2 \, e^{KT} + \ep + 1\,, 
    \end{align*}
where we have used facts that $\|\Lambda X^n(t)\|_2 = \|\Lambda \vec{m}_k(t)/\ell\|_2 = \|1/\ell\|_2 \leq 1$ for all $t$, and that $\ell \geq 1$ for sufficiently large $n$. 
Thus the claim \eqref{eq:stoppingtimestochasticbound} is proved.

Now, we express $X^n(t \wedge \tau) - u_n(t \wedge \tau)$ using their semigroup formulations \eqref{eq: alt alt def X^n} and \eqref{eq:RDSemigroupFormula} to obtain 
\begin{equation*}
    \begin{split}
      &  X^n(t \wedge \tau) - u_n(t \wedge \tau) 
        = e^{ (t \wedge \tau) \calL^{W_n}}[X^n(0) - u_0] \\
        &\qquad + \int_0^{t \wedge \tau} e^{ (t \wedge \tau -s) \calL^{W_n} } [ \Phi(X^n(s)) - \Phi(u_n(s))] ds 
         + Y^n(t \wedge \tau).
    \end{split}
\end{equation*}
Thus,
\begin{align*}
    \|X^n(t \wedge \tau) - u_n(t \wedge \tau)\|_2 &\leq  \left\| Y^n(t \wedge \tau) \right\|_2 + \left\| e^{ (t \wedge \tau ) \calL^{W_n} }[X^n(0) - u_0]  \right\|_2 \\
    &\qquad + \left\| \int_0^{t \wedge \tau} e^{ (t \wedge \tau -s) \calL^{W_n} } [ \Phi(X^n(s)) - \Phi(u_n(s))] ds \right\|_2 
    \\
    &=: \left\| Y^n(t \wedge \tau) \right\|_2 + (\mathrm{I}) + (\mathrm{II}). 
    \label{eq: X^n - u to bound}
\end{align*}
First, \Cref{lem: stability of semigroup} implies that 
\begin{equation} \label{eq: bound on i}
    (\mathrm{I}) \leq \|X^n(0) - u_0 \|_2 \,. 
\end{equation} 
Using \Cref{lem: stability of semigroup} and the fact that $\Phi$ is $K$-Lipschitz, we also obtain the bound 
\begin{align*}
    (\mathrm{II}) 
    &\leq \int_0^{t \wedge \tau} \left\| e^{(t \wedge \tau -s)\calL^{W_n}} \big[\Phi(X^n(s))  - \Phi(u_n(s))\big] \right\|_2 \, ds  
    \\
    &\leq \int_0^{t \wedge \tau} \left\| \Phi(X^n(s))  - \Phi(u_n(s)) \right\|_2 \, ds 
    \\
    &\leq K \int_0^{t \wedge \tau} \|X^n(s) - u_n(s)\|_2 \, ds. \numberthis \label{eq: bound on ii}
\end{align*}
Due to equations \eqref{eq: bound on i} and \eqref{eq: bound on ii}, we have 
\begin{equation*} 
    \begin{split}
    \|X^n(t \wedge \tau) - u_n(t \wedge \tau)\|_2 &\leq \|X^n(0) - u_0 \|_2 + 
    \| Y^n(t \wedge \tau) \|_2 \\
    &\quad + K \int_0^{t \wedge \tau} \|X^n(s) - u_n(s)\|_2 \, ds.
    \end{split}
\end{equation*}
The integral form of Gronwall's inequality implies that 
\begin{equation*}
    \|X^n(t \wedge \tau) - u_n(t \wedge \tau) \|_2 \leq \big( \|X^n(0) - u_0 \|_2 + 
    \| Y^n(t \wedge \tau) \|_2 \big) e^{KT}.
\end{equation*}
Taking the supremum over the interval $[0,T]$, we get 
\begin{equation*}
    \sup_{t\in [0,T]} \|X^n(t \wedge \tau) - u_n(t \wedge \tau) \|_2 \leq \big( \|X^n(0) - u_0 \|_2 + \sup_{t \in [0,T]} 
    \| Y^n(t \wedge \tau) \|_2 \big) e^{KT}.
\end{equation*}

Taking probabilities, we get 
\begin{align*}
        &P\left(\sup_{t\in[0,T]}\|X^n(t \wedge \tau)-u_n(t \wedge \tau)\|_2 \geq \ep\right) \\
        \leq& P\left( \|X^n(0) - u_0\|_2 + \sup_{t \in [0,T]} \| Y^n(t \wedge \tau) \|_2 \geq \ep e^{-KT} \right) \\
        \leq& P\left( \|X^n(0) - u_0\|_2 \geq \frac{\ep e^{-KT}}{2} \right) + P \left( \sup_{t \in [0,T]} \| Y^n(t \wedge \tau) \|_2 \geq \frac{\ep e^{-KT}}{2} \right).\numberthis \label{eq:lln:FinalEstimate}
\end{align*}
By assumption, the first term of the right-hand side of \eqref{eq:lln:FinalEstimate} converges to zero. 
To bound the second term, we invoke \cite[Theorem 1]{kotelenez1982submartingale}, which is a maximal inequality for stochastic convolutions. 
\cite[Remark 1.1]{kotelenez1982submartingale} implies that our evolution operator $e^{t\calL^{W_n}}$ satisfies the assumptions needed in the theorem; specifically, by \Cref{thm:WellDefinedSemigroup}, we have that $e^{t\calL^{W_n}}$ is bounded. Furthermore, $\|Y^n(t)\|_2$ is separable because it is a c\`adl\`ag process. Therefore, \cite[Theorem 1]{kotelenez1982submartingale} implies that 
%\tcb{Let Kotelenez's $\Phi_s$ be the identity operator, $U(t,s)$ be the bounded operator $e^{(t-s)\calL^{W_n}}$, and $M_s$ be the martingale $Z^n(s)$ (defined in equation \eqref{eq: def Z^n}). Then, Kotelenez's $Q_M$ is the quadratic variation $\sum_{0 \leq s \leq T\wedge \tau} \|\Lambda Z^n(s)\|^2_2$. Then, choose $\varepsilon = \ep^2$.} 
    \begin{equation}\label{eq: kotelenez bound}
        P \left( \sup_{0 \leq t \leq T} \| Y^n(t \wedge \tau) \|_2 > \epsilon \right) \leq \frac{\mathrm{exp}(4T)}{  \epsilon^2 } \mathbb{E} \left[ \sum_{0 \leq s \leq T \wedge \tau} \| \Lambda Z^n(s) \|_2^2 \right]. 
    \end{equation}
    We proceed to show that the right-hand-side of \eqref{eq: kotelenez bound} converges to zero as $n\ra\infty$. Using the fact that $\Lambda Z^n = \Lambda X^n$, we have 
    \begin{equation}\label{eq: L^2 of jump Z}
         \| \Lambda Z^n(s) \|_2^2 = \frac{1}{n} \sum_{k = 1}^{n} (\Lambda X^n(k/n,s))^2 = \frac{1}{n \ell^2} \sum_{k = 1}^{n} (\Lambda m_k(s))^2.
    \end{equation}
    Due to \Cref{lma:Martingale2}, we have that 
    \begin{align*}
        \ex\left[ \sum_{0\leq s \leq t \wedge \tau} (\Lambda m_k(s))^2 \right] = \ex\Bigg[ \int_0^{t \wedge \tau} \frac{1}{n} \sum_{\substack{j = 1 \\ j \neq k}}^{n} W_n(k/n,j/n) (m_j(s) + m_k(s) ) ds 
        \\
        \quad \quad - 
        \int_0^{t \wedge \tau} \ell b(m_k(s)/\ell) - \ell d(m_k(s)/\ell) ds \Bigg] \,.
    \end{align*}
    Combining this equation with \eqref{eq: L^2 of jump Z}, and using the fact that $m_k(t) = \ell X^n(x,t)$ for $x\in I_k$, we obtain 
    \begin{equation*}
        \begin{split}
            &\mathbb{E} \left[ \sum_{0 \leq s \leq T \wedge \tau} \| \Lambda Z^n(s) \|_2^2 \right] = \ex\left[ \sum_{0 \leq s \leq T \wedge \tau} \frac{1}{n\ell^2} \sum_{k=1}^n (\Lambda m_k(s))^2 \right] 
            \\
            &\quad \quad + \mathbb{E} \left[ \int_0^{T \wedge \tau} \left( \frac{1}{n^2 \ell^2} \sum_{k=1}^n \sum_{\substack{i=1\\i \neq k}}^{n} W_n(\frac{k}{n},\frac{i}{n}) (m_i(s)+m_k(s)) + \frac{1}{n\ell} \sum_{k = 1}^{n}  (b + d)\left(\frac{m_k(s)}{\ell}\right) \right) \, ds\right] 
            \\
            & \qquad = \frac{1}{\ell}\, \mathbb{E} \bigg[ \int_0^{T \wedge \tau}  \int_0^1 \int_{(0,1) \setminus I_k} W_n(x,y) (X^n(y,s) + X^n(x,s)) dy dx 
            \\
            & \qquad \qquad + \int_0^{T \wedge \tau} \int_0^1 (b + d)(X^n(x,s)) dx \, ds \bigg]
            \\
            & \qquad \leq \frac{1}{\ell}\, \mathbb{E} \bigg[ \int_0^{T \wedge \tau}  \int_0^1 \int_0^1 W_n(x,y) \big|X^n(y,s) + X^n(x,s)\big| dy \,dx 
            \\
            & \qquad \qquad + \int_0^{T \wedge \tau} \int_0^1 \big|(b + d)(X^n(x,s)) \big| dx \, ds \bigg] 
            \\
            &\qquad \leq \frac{1}{\ell} \ex\left[\int_0^{T\wedge \tau}  (2+K)\|X^n(s)\|_1 ds \right] 
        \end{split}
    \end{equation*}
    where the last inequality is obtained using the $p=1$ case of \eqref{eq:DiffusionOperatorBd} and using the fact that $b+d$ is $K$-Lipschitz for some $K$.   
    Using the fact that $\|X^n(s)\|_1 \leq \|X^n(s)\|_2$ and the bound \eqref{eq:stoppingtimestochasticbound} on $\|X^n(s)\|_2$, we obtain
    \[ \frac{1}{\ell} \ex\left[\int_0^{T\wedge \tau}  (2+K) \|X^n(s)\|_1 ds \right] \leq \frac{C(\| u_0 \|_2,K,T)}{\ell} \to 0 \text{ as } n \to \infty.
    \]

Thus the right-hand side of \eqref{eq: equivalent (A)} converges to $0$ as $n \to \infty$, and the proof is complete.
\end{proof}

\section{Conclusion and future work}\label{sec: conclusion}
We defined graphon RD equations as nonlocal RD equations, where the diffusion term is a nonlocal integral operator whose kernel is the accompanying graphon. Suppose that a sequence of graphs (or graphons) $W_n$ converges in cut norm to a limiting graphon $W$. Our first convergence result states that the solutions to the sequence of graph RD equations corresponding to the graphs $W_n$ converge to the solution of the graphon RD equation corresponding to $W$. Our second main convergence result states that a random-walk birth-death particle system converges in probability to the solution of the graphon RD equation. 

Our first convergence result suggests that graphon RD equations can be used to approximate RD equations on large graphs, which include models of population ecology, disease spread, and chemical reactions \cite{cantrell2004spatial, colizza2007reaction, britton1986reaction}. Our second main convergence result suggests that microscopic (a.k.a. agent-based) models can be approximated by the macroscopic RD equation. 

One direction for future work is to extend the convergence results to systems of RD equations. For example, SIR-type models with diffusion consist of a system of three RD equations \cite{cantrell2004spatial}. 

Another direction for future work is to investigate stochastic particle systems on graphons: that is, one could find a graphon limit of the random-walk birth-death process on graphs. \cite{petit2021random} defines a random walk on graphons and shows formally that it is a limit of random walks on an approximating sequence of graphs. They show that the graph random walk converges to the graphon diffusion equation (in a special case of so-called quotient graphs). One could extend their results to include a birth-death process on graphons.

Beyond what is depicted in Figures \ref{fig: heat relations} and \ref{fig: rd relations}, one can consider graphon limits of RWBD$^{W_n}$. A \textit{graphon random walk} was introduced by Petit et al \cite{petit2021random}, who prove that the sequence of graph random walks corresponding to $W_n$ converge (in $L^2$ norm) to the graphon random walk when $W_n$ is a sequence of quotient graphs that converge to $W$. Furthermore, \cite{petit2021random} show that the graphon diffusion equation is the master equation of the graphon random walk. The study of graphon stochastic processes is a topic for future study.  

\section{Appendix}\label{sec: appendix}

In this section, we show that the quantities $\vec{M}(t)$ (equation \eqref{eq: M}) and $Z^n(t)$  (equation \eqref{eq: def Z^n}) are mean-zero martingales. We begin by deriving the infinitesimal generator for the graph random-walk-birth-death process $\vec{m}$ whose transitions are defined in \eqref{eq: def m jumps}, which is restated as follows.  

For brevity, we change the notation $W_n(k/n,i/n)$ to $W_n^{k,i}$. 
\begin{equation}
\left\{\begin{array}{l}
    \vec{m} \rightarrow \vec{m}+e_i-e_k  \quad \text { at rate } \frac{m_k}{n} W_n^{k,i}, \text{ for all } i,k \in [n] \,, \\
    \vec{m} \rightarrow \vec{m}+e_k \quad \quad \quad \text { at rate } \ell b\left(m_k/\ell\right) \,, \\
    \vec{m} \rightarrow \vec{m}-e_k \quad \quad \quad \text { at rate } \ell d \left(m_k/\ell\right)\,.
    \end{array}\right.
\end{equation} 
The following calculations use standard methods for continuous-time stochastic processes (see, e.g., \cite[Chapter 2]{norris1998markov}). Recall that $\vec{m}(t) = (m_1(t), \dots, m_n(t))$ is an $\R^n$-valued process whose $k$th entry, $m_k(t)$, gives the number of particles at node $k$ at time $t$. The infinitesimal generator $A$ of $\vec{m}(t)$ is a linear operator (i.e. it can be represented as an $n \times n$ matrix). For any continuous bounded function $f : \mathbb{R}^n \to \mathbb{R}^n$, the operator $A$ is defined as 
\begin{equation}\label{eq: def A}
    A(ff(\vec{x})) := \lim\limits_{t \to 0} \frac{ \mathbb{E}[ f(\vec{m}(t)) | \vec{m}(0) = \vec{x} ] - f(\vec{x}) }{t} \,. 
\end{equation}
To compute the right-hand side of \eqref{eq: def A}, we first compute the conditional expectation 
\begin{align*}
    \ex[f(\vec{m}(t)) | \vec{m}(0)=\vec{x}] &= t \sum_{k=1}^n \frac{x_k}{n} \sum_{i=1}^n W_n^{k,i} f(\vec{x}+e_i - e_k) 
    \\
    &\quad \quad + \ell t \sum_{k=1}^n b(x_k/\ell) f(\vec{x}+e_k) + \ell t\sum_{k=1}^n d(x_k/\ell)) f(\vec{x}-e_k)
    \\
    &\quad \quad + f(\vec{x}) \left(1- t \sum_{k=1}^n \left[ \frac{x_k}{n} \left(\sum_{i=1}^n W_n^{k,i}\right) + \ell  \big( b(x_k/\ell) + d(x_k/\ell) \big) \right]\right) 
\end{align*}
which is obtained by expanding over the four possible transitions that $\vec{m}$ can undergo within the time interval $[0,t]$: for any node $k$, a particle on node $k$ can jump to a neighbor $i$; a birth or death can occur at node $k$; and no change can occur. From this conditional expectation, the generator \eqref{eq: def A} becomes 
\begin{align*}
    A(f(\vec{x}))
    &= \sum_{k=1}^n \Bigg[ \frac{x_k}{n} \sum_{i=1}^n W_n^{k,i} \Big( f(\vec{x}+e_i - e_k) - f(\vec{x})\Big) 
    \\
    &\quad \quad + \ell b(x_k/\ell) \Big( f(\vec{x}+e_k) - f(\vec{x})\Big)
    + \ell d(x_k/\ell)) \Big( f(\vec{x}-e_k) - f(\vec{x})\Big) \Bigg] \,. 
\end{align*}

In order to describe the quantity $A(\vec{m}(t))$, we take $f$ to be the identity function. For $f = Id$, we have 
\begin{equation*}
    A(\vec{x}) = \sum_{k=1}^n \bigg[ \frac{x_k}{n} \sum_{i=1}^n W_n^{k,i} (e_i - e_k) + \ell e_k \big( b(x_k/\ell) - 
    d(x_k/\ell)\big)  \bigg] \,. 
\end{equation*}
This is the desired generator \eqref{eq: generator A}.

\subsection{\texorpdfstring{$Z^n(t)$ is a martingale}{Znt is a martingale}}\label{appendix: Z^n is a martingale}
To see why $Z^n(t)$ (equation \eqref{eq: def Z^n})
is a mean-zero martingale, 
we first take the dot product of $\vec{M}(t)$ (equation \eqref{eq: M})
with the $j$th standard unit vector:
\begin{equation}\label{eq: jth coordinate of M}
    \vec{M}(t) \cdot e_j = m_j(t \wedge \tau) - m_j(0) - \int_0^{t \wedge \tau} A(\vec{m}(s)) \cdot e_j \, ds.
\end{equation}
Again, let $W_n^{k,i}$ denote $W_n(k/n,i/n)$ for brevity. Using the definition \eqref{eq: generator A} of the generator $A$, for $\vec{x} = (x_1, \dots, x_{n})$ we have 
\begin{align*}
    A(\vec{x}) \cdot e_j 
    &= \frac{1}{n} \sum_{i,k=1}^n x_k W_n^{k,i} (e_i-e_k) \cdot e_j + \ell \sum_{k=1}^n \big( b(x_k/\ell) - d(x_k/\ell) \big) (e_k \cdot e_j) 
    \\
    &= \frac{1}{n} \sum_{i,k} x_k W_n^{k,i} (\delta_{ij} -\delta_{kj})  + \ell \big(  b(x_j/\ell) - d(x_j/\ell)\big)
    \\
    &= \frac{1}{n} \left( \sum_k x_k W_n^{j,k} - \sum_i x_j W_n^{i,j} \right) + \ell \big( b(x_j/\ell) -d(x_j/\ell)\big)
    \\
    &= \frac{1}{n} \left( \sum_i x_i W_n^{j,i} -\sum_i x_j W_n^{i,j} \right) + \ell \big( b(x_j/\ell) -d(x_j/\ell)\big)
    \\
    &= \frac{1}{n} \sum_i W_n^{i,j} (x_i - x_j) + \ell \big( b(x_j/\ell) -d(x_j/\ell)\big) \,. \numberthis \label{eq: jth coord of Ax}
\end{align*}
Thus, the infinitesimal generator $A$ acts on the $j$th coordinate of $\vec{x}$ as the operator $\calL^{W_n}\vec{x}_j + \ell \Phi(x_j/\ell)$. Furthermore, substituting \eqref{eq: jth coord of Ax} into \eqref{eq: jth coordinate of M}, we obtain 
\begin{equation}\label{eq: M dot ej}
    \begin{split}
        &\vec{M}(t) \cdot e_j 
        =  m_j(t \wedge \tau) - m_j(0) \\
        &\qquad - \int_0^{t \wedge \tau} \left( \frac{1}{n} \sum_{i=1}^{n} W_n^{i,j} (m_i(s) - m_j(s)) +   \ell b(m_j(s)/\ell) - \ell d(m_j(s)/\ell \right) ds
    \end{split}
\end{equation}
which is an $\mathcal{F}_t^{n,\ell}$ mean-zero martingale due to Dynkin's formula \cite[Chapter 7.4]{oksendal2013stochastic}. Let $y\in I_i$ and $x \in I_j$, and treat $W_n$ as a piecewise-constant function as defined in \eqref{eq: step-graphon}. Using the definition \eqref{eq: def X^n} of $X^n$, we rewrite equation \eqref{eq: M dot ej} as 

\begin{equation*}
    \begin{split}
        \vec{M}(t) \cdot e_j 
        &= \ell X^n(x,t\wedge \tau) - \ell X^n(x, 0) \\
        &\qquad - \int_0^{t \wedge \tau} \left( \int_0^1 W_n(y,x) (\ell X^n(y,s) - \ell X^n(x, s)) dy +  \ell \Phi(X^n(x, s)) \right) ds\,. 
    \end{split}
\end{equation*}

Because martingale properties hold under scalar multiplication, the quantity 
\begin{equation*}
    \begin{split}
        &\frac{ 1 }{\ell} \vec{M}(t) \cdot e_j
        = X^n(x, t \wedge \tau) - X^n(x,0) \\ & \qquad - \int_0^{t \wedge \tau} \left( \int_0^1 W_n(y,x) (X^n(y,s) - X^n(x,s)) dy + \Phi(X^n(x,s)) \right) ds \,,
    \end{split}
\end{equation*}
where $x\in I_j$, is a mean-zero martingale. As a result, the quantity $Z^n(t) = \sum_{j=1}^n \frac{1}{\ell} \vec{M}(t) \cdot e_j $ (equation \eqref{eq: def Z^n}) is a mean-zero martingale. 

\section*{Acknowledgment}
This work is supported in part by the National Science Foundation through grants DMS-2309245 and DMS-1937254.

\printbibliography

@book{blount1987comparison,
  title={Comparison of a stochastic model of a chemical reaction with diffusion and the deterministic model},
  author={Blount, Douglas James},
  year={1987},
  publisher={The University of Wisconsin--Madison}
}

@article{du2019multiscale,
  title={Multiscale modeling, homogenization and nonlocal effects: Mathematical and computational issues},
  author={Du, Qiang and Engquist, Bjorn and Tian, Xiaochuan},
  journal={arXiv preprint arXiv:1909.00708},
  year={2019}
}

@incollection{bollobas2007metrics, 
address={Cambridge, UK}, 
series={London Mathematical Society Lecture Note Series}, 
title={Metrics for sparse graphs}, 
DOI={10.1017/CBO9781107325975.009}, 
booktitle={Surveys in Combinatorics 2009}, 
publisher={Cambridge University Press}, 
author={Bollob\'{a}s, B\'{e}la and Riordan, Oliver}, 
editor={Huczynska, Sophie and Mitchell, James D. and Roney-Dougal, Colva M.}, 
year={2009}, 
pages={211--288}, 
collection={London Mathematical Society Lecture Note Series}}

@article{borgs2018LpII,
  title={An {$L^p$} theory of sparse graph convergence {II}: {LD} convergence, quotients and right convergence},
  author={Borgs, Christian and Chayes, Jennifer T. and Cohn, Henry and Zhao, Yufei},
  journal = {The Annals of Probability},
  volume = {46},
  number = {1},
  pages = {337--396},
  year={2018}
}

@article{borgs2019LpI,
  title={An {$L^p$} theory of sparse graph convergence {I}: {L}imits, sparse random graph models, and power law distributions},
  author={Borgs, Christian and Chayes, Jennifer and Cohn, Henry and Zhao, Yufei},
  journal={Transactions of the American Mathematical Society},
  volume={372},
  number={5},
  pages={3019--3062},
  year={2019}
}

@article{merris1994laplacian,
  title={Laplacian matrices of graphs: a survey},
  author={Merris, Russell},
  journal={Linear Algebra and its Applications},
  volume={197},
  pages={143--176},
  year={1994},
  publisher={Elsevier}
}

@article{zhang2024ginzburg,
  title={Ginzburg--Landau Functionals in the Large-Graph Limit},
  author={Zhang, Edith and Scott, James and Du, Qiang and Porter, Mason A},
  journal={Pure and Applied Functional Analysis}, 
  volume={1},
  pages={169--209}, 
  %arXiv preprint arXiv:2408.00422}

@inproceedings{yuasa1998internal,
  title={Internal observation systems and a theory of reaction-diffusion equation on a graph},
  author={Yuasa, Hideo and Ito, Masami},
  booktitle={1998 IEEE International Conference on Systems, Man, and Cybernetics},
  volume={4},
  pages={3669--3673},
  year={1998},
  organization={IEEE}
}

@article{serfaty2011gamma,
  title={Gamma-convergence of gradient flows on {H}ilbert and metric spaces and applications},
  author={Serfaty, Sylvia},
  journal={Discrete and Continuous Dynamical Systems},
  volume={31},
  number={4},
  pages={1427--1451},
  year={2011}
}

@book{norris1998markov,
  title={Markov chains},
  author={Norris, James R},
  number={2},
  year={1998},
  publisher={Cambridge University Press}
}

@book{oksendal2013stochastic,
  title={Stochastic differential equations: an introduction with applications},
  author={Oksendal, Bernt},
  year={2013},
  publisher={Springer Science \& Business Media}
}

@book{lovasz2012large,
  title={Large networks and graph limits},
  author={Lov{\'a}sz, L{\'a}szl{\'o}},
  volume={60},
  year={2012},
  publisher={American Mathematical Society}
}

@article{lovasz2007szemeredi,
  title={Szemer{\'e}di’s lemma for the analyst},
  author={Lov{\'a}sz, L{\'a}szl{\'o} and Szegedy, Bal{\'a}zs},
  journal={GAFA Geometric And Functional Analysis},
  volume={17},
  pages={252--270},
  year={2007},
  publisher={Springer}
}

@book{janson2010graphons,
  title={Graphons, Cut Norm and Distance, Couplings and Rearrangements},
  author={Janson, Svante},
  volume = {4},
  series = {NYJM Monographs},
  publisher={SUNY at Albany},
  year={2013}
}

@book{el2023nonlocal,
  title={Nonlocal Continuum Limits of $p$-Laplacian Problems on Graphs},
  author={El Bouchairi, Imad and Fadili, Jalal and Hafiene, Yosra and Elmoataz, Abderrahim},
  year={2023},
  publisher={Cambridge University Press}
}

@article{bramburger2023pattern,
  title={Pattern formation in random networks using graphons},
  author={Bramburger, Jason and Holzer, Matt},
  journal={SIAM Journal on Mathematical Analysis},
  volume={55},
  number={3},
  pages={2150--2185},
  year={2023},
  publisher={Society for Industrial and Applied Mathematics}
}

@article{heinze2024graph,
  title={Graph-based nonlocal gradient systems and their local limits},
  author={Heinze, Georg},
  year={2024},
  publisher={Universit{\"a}t Augsburg}
}

@book{britton1986reaction,
  title={Reaction-diffusion equations and their applications to biology.},
  author={Britton, Nicholas F},
  year={1987},
  publicher={Academic Press}
}

@article{de1986reaction,
  title={Reaction-diffusion equations for interacting particle systems},
  author={De Masi, Anna and Ferrari, Pablo A and Lebowitz, Joel L},
  journal={Journal of Statistical Physics},
  volume={44},
  number={3},
  pages={589--644},
  year={1986},
  publisher={Springer}
}

@article{van2021theory,
  title={A theory of pattern formation for reaction--diffusion systems on temporal networks},
  author={Van Gorder, Robert A},
  journal={Proceedings of the Royal Society A},
  volume={477},
  number={2247},
  pages={1--28},
  year={2021},
  publisher={The Royal Society Publishing}
}

@article{scalise2016emulating,
  title={Emulating cellular automata in chemical reaction--diffusion networks},
  author={Scalise, Dominic and Schulman, Rebecca},
  journal={Natural Computing},
  volume={15},
  pages={197--214},
  year={2016},
  publisher={Springer}
}

@article{colizza2007reaction,
  title={Reaction-diffusion processes and metapopulation models in heterogeneous networks},
  author={Colizza, Vittoria and Pastor-Satorras, Romualdo and Vespignani, Alessandro},
  journal={Nature Physics},
  volume={3},
  pages={276--282},
  year={2007},
  publisher={Nature Publishing Group}
}

@article{vanag2009cross,
  title={Cross-diffusion and pattern formation in reaction--diffusion systems},
  author={Vanag, Vladimir K and Epstein, Irving R},
  journal={Physical Chemistry Chemical Physics},
  volume={11},
  number={6},
  pages={897--912},
  year={2009},
  publisher={Royal Society of Chemistry}
}

@book{cantrell2004spatial,
  title={Spatial ecology via reaction-diffusion equations},
  author={Cantrell, Robert Stephen and Cosner, Chris},
  year={2004},
  publisher={John Wiley \& Sons}
}

@article{du2021maximum,
  title={Maximum bound principles for a class of semilinear parabolic equations and exponential time-differencing schemes},
  author={Du, Qiang and Ju, Lili and Li, Xiao and Qiao, Zhonghua},
  journal={SIAM Review},
  volume={63},
  number={2},
  pages={317--359},
  year={2021},
  publisher={SIAM}
}

@article{luo2017convergence,
  title={Convergence of the graph Allen--Cahn scheme},
  author={Luo, Xiyang and Bertozzi, Andrea L},
  journal={Journal of Statistical Physics},
  volume={167},
  pages={934--958},
  year={2017},
  publisher={Springer}
}

@article{du1992analysis,
  title={Analysis and approximation of the Ginzburg--Landau model of superconductivity},
  author={Du, Qiang and Gunzburger, Max D and Peterson, Janet S},
  journal={Siam Review},
  volume={34},
  number={1},
  pages={54--81},
  year={1992},
  publisher={SIAM}
}

@book{du2019nonlocal,
  title={Nonlocal Modeling, Analysis, and Computation},
  author={Du, Qiang},
  year={2019},
  publisher={SIAM}
}

@article{funaki2019motion,
  title={Motion by mean curvature from Glauber--Kawasaki dynamics},
  author={Funaki, Tadahisa and Tsunoda, Kenkichi},
  journal={Journal of Statistical Physics},
  volume={177},
  pages={183--208},
  year={2019},
  publisher={Springer}
}

@article{kotelenez1982submartingale,
  title={A submartingale type inequality with applications to stochastic evolution equations},
  author={Kotelenez, Peter},
  journal={Stochastics: An International Journal of Probability and Stochastic Processes},
  volume={8},
  number={2},
  pages={139--151},
  year={1982},
  publisher={Taylor \& Francis}
}

@book{engel2000one,
  title={One-parameter semigroups for linear evolution equations},
  author={Engel, Klaus-Jochen and Nagel, Rainer and Brendle, Simon},
  volume={194},
  year={2000},
  publisher={Springer}
}

@article{watanabe2022continuum,
  title={Continuum limit of nonlocal diffusion on inhomogeneous networks},
  author={Watanabe, Itsuki},
  journal={Journal of Dynamics and Differential Equations},
  pages={1--20},
  year={2022},
  publisher={Springer}
}

@article{medvedev2014nonlinear,
  title={The nonlinear heat equation on dense graphs and graph limits},
  author={Medvedev, Georgi S},
  journal={SIAM Journal on Mathematical Analysis},
  volume={46},
  number={4},
  pages={2743--2766},
  year={2014},
  publisher={SIAM}
}

@article{angstmann2013pattern,
  title={Pattern formation on networks with reactions: A continuous-time random-walk approach},
  author={Angstmann, Christopher N and Donnelly, Isaac C and Henry, Bruce I},
  journal={Physical Review E—Statistical, Nonlinear, and Soft Matter Physics},
  volume={87},
  number={3},
  pages={032804},
  year={2013},
  publisher={American Physical Society}
}

@article{petit2021random,
  title={Random walks on dense graphs and graphons},
  author={Petit, Julien and Lambiotte, Renaud and Carletti, Timoteo},
  journal={SIAM Journal on Applied Mathematics},
  volume={81},
  number={6},
  pages={2323--2345},
  year={2021},
  publisher={SIAM}
}

@book{pazy2012semigroups,
  title={Semigroups of linear operators and applications to partial differential equations},
  author={Pazy, Amnon},
  volume={44},
  year={2012},
  publisher={Springer Science \& Business Media}
}

\end{document}